\documentclass{amsart}
\usepackage{amsmath,amsthm,amssymb,amsfonts,amscd, array}

\newtheorem{Theorem}{Theorem}[section]
\newtheorem{defn}[Theorem]{Definition}

\newtheorem{lemma}[Theorem]{Lemma}

\title[Complex Intuitionistic fuzzy bracket product]
{Complex Intuitionistic fuzzy bracket product}
\author[Ameer Jaber]{}
 \email{ameerj@hu.edu.jo}
 \medskip
 \author[Rania Shaqbou'a]{}
 \email{Rania.Shaqboua@hu.edu.jo}

\begin{document}

\maketitle

\centerline{\scshape  Ameer Jaber}
\medskip

{\footnotesize \centerline{ Department of Mathematics }
\centerline{ The Hashemite University} \centerline{ Zarqa 13115, Jordan
} }
\bigskip
\centerline{\scshape Rania Shaqbou'a}
\medskip
{\footnotesize \centerline{ Department of Mathematics }
\centerline{ The Hashemite University} \centerline{ Zarqa 13115, Jordan
} }
\begin{abstract}
A complex intuitionistic fuzzy Lie superalgebra is a generalization of intuitionistic fuzzy Lie superalgebra whose membership function takes values in the unit disk in the complex plane. In \cite{AMR}, we introduced and studied the concepts of complex intuitionistic fuzzy Lie sub-superalgebras and complex intuitionistic fuzzy ideals of Lie superalgebras. Moreover, we also, in \cite{AMR}, defined the image and preimage of complex intuitionistic fuzzy Lie sub-superalgebra under Lie superalgebra anti-homomorphism, and he investigated the properties of anti-complex intuitionistic fuzzy Lie sub-superalgebras and anti-complex intuitionistic fuzzy ideals under anti-homomorphisms of Lie superalgebras. In this research, we use the concepts of complex intuitionistic fuzzy Lie superalgebras to introduce the complex intuitionistic fuzzy bracket products. Finally, we use the definitions of the image and preimage of complex intuitionistic fuzzy ideals under Lie superalgebra anti-homomorphisms, to study the characterizations of the image and preimage of the complex intuitionistic fuzzy bracket products under Lie superalgebra anti-homomorphisms.
\end{abstract}
\section{Introduction}
The notion of intuitionistic fuzzy sets was introduced by Atanassov (see \cite{A0}). He presented in \cite{A0} the idea of intuitionistic fuzzy sets. He also in \cite{A} defined some properties of intuitionistic fuzzy sets. Atanassov presented in \cite{A1} interesting new operations about intuitionistic fuzzy sets. An intuitionistic fuzzy set is the generalization of fuzzy set. Recently, Biswas applied the concepts of intuitionistic fuzzy sets to the theory of groups and studied intuitionistic fuzzy subgroups of a group (see \cite{RB}); also Banerjee studied intuitionistic fuzzy subrings and ideals of a ring (see \cite{BB}). Moreover, Jun investigated the concept of intuitionistic nil-radicals of intuitionistic fuzzy ideals in rings (see \cite{YB}) and Davvaz, Dudek and Jun applied the notion of intuitionistic fuzzy sets to certain types of modules (see \cite{BD}). Then in \cite{CZ} W. chen and S. Zhang introduced the concept of intuitionistic fuzzy Lie superalgebras and intuitionistic fuzzy ideals. It is known that fuzzy sets are intuitionistic fuzzy sets but the converse is not necessarily true (for more details see \cite{A}). More recently, Alkouri and Salleh gave in \cite{AS} the idea of complex intuitionistic fuzzy subsets and then they enlarge the basic properties of it. This concept became more effective and useful in scientific field because it deals with degree of membership and nonmembership in complex plane. They also initiated the concept of complex intuitionistic fuzzy relation and developed fundamental operation of complex intuitionistic fuzzy sets in \cite{AS1, AS2}. Then Garg and Rani made in \cite{GR} a huge effort to generalize the notion of complex intuitionistic fuzzy sets in decision-making problems.\\
In \cite{SS-1}, S. Shaqaqha introduced the concepts of complex fuzzy sets to the theory of Lie algebras and studied complex fuzzy Lie subalgebras. Furthermore, in \cite{SS-2, SS-3} S. Shaqaqha and M. Al-Deiakeh introduced the concepts of complex intuitionistic fuzzy Lie algebras and complex intuitionistic fuzzy Lie ideals and they studied the relation between complex intuitionistic fuzzy Lie subalgebras (ideals) and intuitionistic fuzzy Lie subalgebras (ideals).\\
Following the same approach as in \cite{AMR}, we start this research as follows. In section 2, we recall some basic definitions and notions which will be used in what follows. In section 3, we introduce the definitions of $\mathbb{Z}_2$-graded complex intuitionistic fuzzy  vector subspaces, and complex intuitionistic fuzzy ideals, then we study the conditions on complex intuitionistic fuzzy bracket products to be $\mathbb{Z}_2$-graded complex intuitionistic fuzzy subspaces, and complex intuitionistic fuzzy ideals. Finally, in section 4, we discuss the properties of the images and preimages of complex intuitionistic fuzzy bracket products under Lie superalgebra anti-homomorphisms.
\section{Complex intuitionistic fuzzy sets}
Let $X\not=\phi$. A complex intuitionistic fuzzy set on $X$ is an object having the form $A=\{(x,\lambda_A(x),\rho_A(x))\ | x\in X\}$, where the complex functions $\lambda_A : X\rightarrow \mathbb{C}$ and $\rho_A : X\rightarrow \mathbb{C}$ denote the degree of membership (namely $\lambda_A(x)$) and the degree of non-membership (namely $\rho_A(x)$) of each element $x\in X$ to the set $A$, respectively, that assign to any element $x\in X$ complex numbers $\lambda_A(x)$, $\rho_A(x)$ lie within the unit circle with the property $|\lambda_A(x)|+|\rho_A(x)|\leq 1$.
For the sake of simplicity, we shall use the symbol $A=(\lambda_A, \rho_A)$ for the complex intuitionistic fuzzy set $A=\{(x,\lambda_A(x),\rho_A(x))\ | x\in X\}$.\\
We shall assume $\lambda_A(x)$, $\rho_A(x)$ will be represented by $r_A(x)e^{i2\pi \omega_A(x)}$ and $\hat{r}_A(x)e^{i2\pi \hat{\omega}_A(x)}$, respectively, where $i=\sqrt{-1} $, $r_A(x), \hat{r}_A(x), \omega_A(x), \hat{\omega}_A(x) \in [0,1]$. Thus the property of $|\lambda_A(x)|+|\rho_A(x)|\leq 1$ implies $r_A(x)+\hat{r}_A(x)\leq1$. Note that the intuitionistic fuzzy set is a special case of complex intuitionistic fuzzy set with $\omega_A(x)=\hat{\omega}_A(x)=0.$
Also, if $\rho_A(x)=(1-r_A(x))e^{i2\pi(1-\omega_A(x))}$, then we obtain a complex fuzzy set.
Let $\alpha e^{i2\pi\beta}$ and $\gamma e^{i2\pi\delta}$ be two complex numbers, where $\alpha, \beta, \gamma, \delta\in [0,1]$. By $\alpha e^{i2\pi\beta}\leq \gamma e^{i2\pi\delta}$ we mean $\alpha\leq\gamma$ and $\beta\leq\delta$.
In this paper,we use the symbols $a\wedge b=min\{a,b\}$ and $a\vee b=max\{a,b\}$.
Let $A=(\lambda_A,\rho_A)$ be a complex intuitionistic fuzzy set on $X$ with the degree of membership $\lambda_A(x)=r_A(x)e^{i2\pi \omega_A(x)}$ and the degree of non-membership $\rho_A(x)=\hat{r}_A(x)e^{i2\pi\hat{\omega}_A(x)}$. Then $A$ is said to be a homogeneous complex intuitionistic fuzzy set if the following two conditions hold $\forall x,y\in X$\\
(1) $r_A(x)\leq r_A(y)$ if and only if $\omega_A(x)\leq\omega_A(y)$,\\
(2) $\hat{r}_A(x)\leq \hat{r}_A(y)$ if and only if $\hat{\omega}_A(x)\leq\hat{\omega}_A(y)$.\\
Let $A=(\lambda_A,\rho_A)$ and $B=(\lambda_B,\rho_B)$ be two complex intuitionistic fuzzy sets on the same set $X$, we say that $A$ is homogeneous with $B$ if the following conditions hold $\forall x,y\in X$\\
(1) $r_A(x)\leq r_B(y)$ if and only if $\omega_A(x)\leq\omega_B(y)$,\\
(2) $\hat{r}_A(x)\leq \hat{r}_B(y)$ if and only if $\hat{\omega}_A(x)\leq\hat{\omega}_B(y)$.\\
\begin{defn}
Let $K$ be any field, and let $\mathbf{V}$ be a $K$-vector space. A complex intuitionistic fuzzy (CIF for short) set on $\mathbf{V}$ defined as an object having the form $A=\{(x,\lambda_A(x),\rho_A(x))\ | x\in \mathbf{V}\}$, where the complex functions $\lambda_A : \mathbf{V}\rightarrow \mathbb{C}$ and $\rho_A : \mathbf{V}\rightarrow \mathbb{C}$ denote the degree of membership (namely $\lambda_A(x)$) and the degree of non-membership (namely $\rho_A(x)$) of each element $x\in \mathbf{V}$ to the set $A$, respectively, that assign to any element $x\in \mathbf{V}$ complex numbers $\lambda_A(x)$, $\rho_A(x)$ lie within the unit circle with the property $|\lambda_A(x)|+|\rho_A(x)|\leq 1$.
\end{defn}
We shall use the symbol $A=(\lambda_A, \rho_A)$ for the CIF set $A=\{(x,\lambda_A(x),\rho_A(x))\ | x\in \mathbf{V}\}$.
\begin{defn}
Let $A=(\lambda_A,\rho_A)$ and $B=(\lambda_B,\rho_B)$ be CIF sets of a vector subspace $\mathbf{V}$. Then $A\subseteq B$ if $\lambda_A(x)\leq\lambda_B(x)$ and $\rho_A(x)\geq\rho_B(x)$ for all $x\in \mathbf{V}$.
\end{defn}
\begin{defn}
Let $\mathbf{V}$ be a $K$-vector space. A CIF set $A=(\lambda_A, \rho_A)$ of a vector space $\mathbf{V}$ is called a CIF vector space of $\mathbf{V}$, if it satisfies the following conditions\\
for any $x,y\in\mathbf{V}$, $\alpha\in K$\\
(1) $\lambda_A(x+y)\geq \lambda_A(x)\wedge\lambda_A(y)$, and $\rho_A(x+y)\leq \rho_A(x)\vee\rho_A(y)$\\
(2) $\lambda_A(\alpha x)\geq \lambda_A(x)$, and $\rho_A(\alpha x)\leq \rho_A(x)$.
\end{defn}
From this definition, we know that for any $x\in\mathbf{V}$, $\lambda_A(0)\geq\lambda_A(x)$ and $\rho_A(0)\leq\rho_A(x)$. In this paper, we always assume that $\lambda_A(0)=1e^{i2\pi}=1$ and $\rho_A(0)=0e^{i2(\pi)0}=0$.
\begin{defn}\cite{AMR}\label{def-1}
Let $A=(\lambda_A,\rho_A)$ and $B=(\lambda_B,\rho_B)$ be CIF vector subspaces of a vector subspace $\mathbf{V}$, where $\lambda_A=r_Ae^{i2\pi\omega_A},\lambda_B=r_Be^{i2\pi\omega_B}$and $\rho_A=\hat{r}_Ae^{i2\pi\hat{\omega}_A}, \rho_B=\hat{r}_Be^{i2\pi\hat{\omega}_B}$. If $A$ is homogenous with $B$. Then the complex intuitionistic sum of $A=(\lambda_A,\rho_A)$ and $B=(\lambda_B,\rho_B)$ is defined to the CIF set $A+B=(\lambda_{A+B},\rho_{A+B})$ of $\mathbf{V}$ given by
$$\lambda_{A+B}(x)=\begin{cases} \sup\limits_{x=a+b}\{(r_A(a)\wedge r_B(b))\}e^{i2\pi\sup\limits_{x=a+b}\{ (\omega_A(a)\wedge\omega_B(b))\}} & :\quad {\rm if}\ x=a+b\\
0 & :\quad otherwise,\end{cases}$$
$$\rho_{A+B}(x)=\begin{cases} \inf\limits_{x=a+b}\{(\hat{r}_A(a)\vee\hat{r}_B(b))\}e^{i2\pi\inf\limits_{x=a+b}\{ (\hat{\omega}_A(a)\vee\hat{\omega}_B(b))\}} & :\quad {\rm if}\ x=a+b\\
1 & :\quad otherwise,\end{cases}$$
\end{defn}
Further, if $A\cap B=(\lambda_{A\cap B},\rho_{A\cap B})$, where
$$\lambda_{A\cap B}(x)=\begin{cases} 0 & :\quad x\not=0\\
1 & :\quad x=0,\end{cases}\ \ \ {\rm and}\ \ \
\rho_{A\cap B}(x)=\begin{cases} 1 & :\quad x\not=0\\
0 & :\quad x=0,\end{cases}.$$ Then $A+B$ is said to be the direct sum and denoted by $A\oplus B$.
\begin{lemma}\cite{AMR}\label{mylemma-1}\textnormal{
Let $A=(\lambda_A,\rho_A)$ and $B=(\lambda_B,\rho_B)$ be CIF vector subspaces of a vector space $\mathbf{V}$ such that $A$ is homogenous with $B$. Then $A+B=(\lambda_{A+B},\rho_{A+B})$ is also a CIF vector subspaces of $\mathbf{V}$.}
\end{lemma}
\begin{defn}\cite{AMR}
Let $A=(\lambda_A, \rho_A)$ be a CIF vector subspace of a $K$-vector space $\mathbf{V}$. For $\alpha\in K$ and $x\in \mathbf{V}$, define $\alpha A=(\lambda_{\alpha A}, \rho_{\alpha A})$, where $$\lambda_{\alpha A}(x)=\begin{cases} \lambda_A(\alpha^{-1}x)=r_A(\alpha^{-1}x)e^{i2\pi\omega_A(\alpha^{-1}x)} & :\quad \alpha\not=0\\
1 & :\quad \alpha=0, x=0\\
0 & :\quad \alpha=0, x\not=0\end{cases}$$ and
$$\rho_{\alpha A}(x)=\begin{cases} \rho_A(\alpha^{-1}x)=\hat{r}_A(\alpha^{-1}x)e^{i2\pi\hat{\omega}_A(\alpha^{-1}x)} & :\quad \alpha\not=0\\
0 & :\quad \alpha=0, x=0\\
1 & :\quad \alpha=0, x\not=0\end{cases}$$
\end{defn}
\section{Complex intuitionistic fuzzy brackets}
In this section we assume that $\mathbf{V}$ is a $\mathbb{Z}_2$-graded vector space over a field $K$. Also we assume that $A$ is homogeneous with $B$ for any two complex intuitionistic fuzzy sets $A=(\lambda_A,\rho_A)$, $B=(\lambda_B,\rho_B)$ of $\mathbf{V}$.
\begin{defn}\cite{AMR}\label{def-2}
Let $\mathbf{V}=\mathbf{V}_0+\mathbf{V}_1$ be a $\mathbb{Z}_2$-graded vector space. Suppose that $A_0=(\lambda_{A_0}, \rho_{A_0})$ and $A_1=(\lambda_{A_1}, \rho_{A_1})$ are CIF vector subspaces of $\mathbf{V}_0$ and $\mathbf{V}_1$, respectively. Define $\mathfrak{a}_0=(\lambda_{\mathfrak{a}_0}, \rho_{\mathfrak{a}_0})$
where
$$\lambda_{\mathfrak{a}_0}(x)=\begin{cases} \lambda_{A_0}(x) & :\quad x\in \mathbf{V}_0\\
0 & :\quad x\not\in \mathbf{V}_0\end{cases}\ \ and\ \
\rho_{\mathfrak{a}_0}(x)=\begin{cases} \rho_{A_0}(x) & :\quad x\in \mathbf{V}_0\\
1 & :\quad x\not\in \mathbf{V}_0\end{cases}$$ and define $\mathfrak{a}_1=(\lambda_{\mathfrak{a}_1}, \rho_{\mathfrak{a}_1})$
where
$$\lambda_{\mathfrak{a}_1}(x)=\begin{cases} \lambda_{A_1}(x) & :\quad x\in \mathbf{V}_1\\
0 & :\quad x\not\in \mathbf{V}_1\end{cases}\ \ and\ \
\rho_{\mathfrak{a}_1}(x)=\begin{cases} \rho_{A_1}(x) & :\quad x\in \mathbf{V}_1\\
1 & :\quad x\not\in \mathbf{V}_1\end{cases}$$
\end{defn}
Then $\mathfrak{a}_0=(\lambda_{\mathfrak{a}_0}, \rho_{\mathfrak{a}_0})$ and $\mathfrak{a}_1=(\lambda_{\mathfrak{a}_1}, \rho_{\mathfrak{a}_1})$ are the CIF vector subspaces of $\mathbf{V}$. Moreover, we have $\mathfrak{a}_0\cap \mathfrak{a}_1=(\lambda_{\mathfrak{a}_0\cap\mathfrak{a}_1}, \rho_{\mathfrak{a}_0\cap\mathfrak{a}_1})$, where $$\lambda_{\mathfrak{a}_0\cap \mathfrak{a}_1}(x)=\lambda_{\mathfrak{a}_0}(x)\wedge\lambda_{\mathfrak{a}_1}(x)=\begin{cases} 1 & :\quad x=0\\
0 & :\quad x\not=0,\end{cases}$$ and $$\rho_{\mathfrak{a}_0\cap \mathfrak{a}_1}(x)=\rho_{\mathfrak{a}_0}(x)\vee\rho_{\mathfrak{a}_1}(x)=\begin{cases} 0 & :\quad x=0\\ 1 & :\quad x\not=0.\end{cases}$$ So $\mathfrak{a}_0+\mathfrak{a}_1$ is the direct sum and is denoted by $A_0\oplus A_1$. If $A=(\lambda_A, \rho_A)$ is a CIF vector subspace of $\mathbf{V}$ and $A=A_0\oplus A_1$, then $A=(\lambda_A, \rho_A)$ is called a $\mathbb{Z}_2$-graded CIF vector subspaces of $\mathbf{V}$.
\begin{defn}\cite{AMR}
Let $\mathbf{V}=\mathbf{V}_0+\mathbf{V}_1$ be a $\mathbb{Z}_2$-graded vector space, and let $A=(\lambda_A,\rho_A)$  be CIF set of $\mathbf{V}$. Then $A=(\lambda_A,\rho_A)$ is called a CIF ideal of $\mathbf{V}$, if it satisfies the following conditions:\\
(1) $A=(\lambda_A,\rho_A)$ is a $\mathbb{Z}_2$-graded CIF vector subspace of $\mathbf{V}$,\\
(2) $\lambda_A([x,y])\geq\lambda_A(x)\vee \lambda_A(y)$ and $\rho_A([x,y])\leq\rho_A(x)\wedge\rho_A(y)$.
\end{defn}
\begin{Theorem}\cite{AMR}
If $A=(\lambda_A, \rho_A)$ and $B=(\lambda_B, \rho_B)$ are CIF ideals of $\mathbf{V}=\mathbf{V}_0+\mathbf{V}_1$, then so is $A+B=(\lambda_{A+B}, \rho_{A+B})$.
\end{Theorem}
\begin{defn}
Let $A=(\lambda_A,\rho_A)$ and $B=(\lambda_B,\rho_B)$ be CIF sets of $\mathbf{V}$ such that $A$ is homogenous with $B$. We define the complex intuitionistic fuzzy bracket product $[A,B]=(\lambda_{[A,B]},\rho_{[A,B]})$ where
$$\lambda_{[A,B]}(x)=\begin{cases} \sup\limits_{x=\sum\limits_{i\in N}\alpha_i[x_i,y_i]}\{\min\limits_{i\in N}\{r_A(x_i)\wedge r_B(y_i)\}e^{i2\pi\min\limits_{i\in N}\{\omega_A(x_i)\wedge\omega_B(y_i)\}}\} & :\quad {\rm where}\ \alpha_i\in K,x_i,y_i\in\mathbf{V}\\
0 & :\quad {\rm if}\ x\not=\sum\limits_{i\in N}\alpha_i[x_i,y_i],\end{cases}$$ and
$$\rho_{[A,B]}(x)=\begin{cases} \inf\limits_{x=\sum\limits_{i\in N}\alpha_i[x_i,y_i]}\{\max\limits_{i\in N}\{\hat{r}_A(x_i)\vee\hat{r}_B(y_i)\}e^{i2\pi\max\limits_{i\in N}\{\hat{\omega}_A(x_i)\vee\hat{\omega}_B(y_i)\}}\} & :\quad {\rm where}\ \alpha_i\in K,x_i,y_i\in \mathbf{V}\\
1 & :\quad {\rm if}\ x\not=\sum\limits_{i\in N}\alpha_i[x_i,y_i]\end{cases}.$$
\end{defn}
We have to remark that if $x,y\in\mathbf{V}$, then\\
(1) $\lambda_{[A,B]}([x,y])=r_{[A,B]}([x,y])e^{i2\pi\omega_{[A,B]}([x,y])}$ and $r_{[A,B]}([x,y])\geq r_A(x)\wedge r_B(y)$, $\omega_{[A,B]}([x,y])\geq\omega_A(x)\wedge\omega_B(y)$\\
(2) $\rho_{[A,B]}([x,y])=\hat{r}_{[A,B]}([x,y])e^{i2\pi\hat{\omega}_{[A,B]}([x,y])}$ and $\hat{r}_{[A,B]}([x,y])\leq r_A(x)\vee\hat{r}_B(y)$, $\hat{\omega}_{[A,B]}([x,y])\leq\hat{\omega}_A(x)\vee\hat{\omega}_B(y)$.
\begin{lemma}\label{lem-2}\textnormal{
Let $A_1=(\lambda_{A_1},\rho_{A_1})$, $A_2=(\lambda_{A_2},\rho_{A_2})$ and $B=(\lambda_B,\rho_B)$ be CIF sets of $\mathbf{V}$ such that $A_1\subseteq B$, $A_2\subseteq B$, then $A_1+A_2\subseteq B$.}
\end{lemma}
\begin{proof}
Let $x\in\mathbf{V}$. Then
\begin{align*}
\lambda_{A_1+A_2}(x)&=r_{A_1+A_2}(x)e^{i2\pi\omega_{A_1+A_2}(x)}\\
&\leq\sup_{x=a+b}\{(r_{A_1}(a)\wedge r_{A_2}(b))e^{i2\pi(\omega_{A_1}(a)\wedge\omega_{A_2}(b))}\}\\
&\leq\sup_{x=a+b}\{(r_B(a)\wedge r_B(b))e^{i2\pi(\omega_B(a)\wedge\omega_B(b))}\}\\
&\leq\sup_{x=a+b}\{r_B(a+b)e^{i2\pi\omega_B(a+b)}\}\\
&=r_B(x)e^{i2\pi\omega_B(x)}=\lambda_B(x)
\end{align*}
and
\begin{align*}
\rho_{A_1+A_2}(x)&=\hat{r}_{A_1+A_2}(x)e^{i2\pi\hat{\omega}_{A_1+A_2}(x)}\\
&\geq\inf_{x=a+b}\{(\hat{r}_{A_1}(a)\vee \hat{r}_{A_2}(b))e^{i2\pi(\hat{\omega}_{A_1}(a)\vee\hat{\omega}_{A_2}(b))}\}\\
&\geq\inf_{x=a+b}\{(\hat{r}_B(a)\vee\hat{r}_B(b))e^{i2\pi(\hat{\omega}_B(a)\vee\hat{\omega}_B(b))}\}\\
&\geq\inf_{x=a+b}\{\hat{r}_B(a+b)e^{i2\pi\hat{\omega}_B(a+b)}\}\\
&=\hat{r}_B(x)e^{i2\pi\hat{\omega}_B(x)}=\rho_B(x).
\end{align*}
Hence $A_1+A_2\subseteq B$.
\end{proof}
\begin{lemma}\label{lem-1}\textnormal{
Let $A_1=(\lambda_{A_1},\rho_{A_1})$, $A_2=(\lambda_{A_2},\rho_{A_2})$ and $B_1=(\lambda_{B_1},\rho_{B_1})$, $B_2=(\lambda_{B_2},\rho_{B_2})$ be CIF sets of $\mathbf{V}$ such that $A_1\subseteq A_2$, $B_1\subseteq B_2$. Then $[A_1,B_1]\subseteq[A_2,B_2]$. In particular, if $A=(\lambda_A,\rho_A)$, $B=(\lambda_B,\rho_B)$ are CIF sets of $\mathbf{V}$, then $[A_1,B]\subseteq[A_2,B]$ and $[A,B_1]\subseteq[A,B_2]$.}
\end{lemma}
\begin{proof}
Let $x\in\mathbf{V}$. Then
\begin{align*}
\lambda_{[A_1,B_1]}(x)&=\sup\limits_{x=\sum\limits_{i\in N}\alpha_i[x_i,y_i]}\{\min\limits_{i\in N}\{r_{A_1}(x_i)\wedge r_{B_1}(y_i)\}e^{i2\pi\min\limits_{i\in N}\{\omega_{A_1}(x_i)\wedge\omega_{B_1}(y_i)\}}\}\\
&\leq\sup\limits_{x=\sum\limits_{i\in N}\alpha_i[x_i,y_i]}\{\min\limits_{i\in N}\{r_{A_2}(x_i)\wedge r_{B_2}(y_i)\}e^{i2\pi\min\limits_{i\in N}\{\omega_{A_2}(x_i)\wedge\omega_{B_2}(y_i)\}}\}=\lambda_{[A_2,B_2]}(x)
\end{align*}
and
\begin{align*}
\rho_{[A_1,B_1]}(x)&=\inf\limits_{x=\sum\limits_{i\in N}\alpha_i[x_i,y_i]}\{\max\limits_{i\in N}\{\hat{r}_{A_1}(x_i)\vee\hat{r}_{B_1}(y_i)\}e^{i2\pi\max\limits_{i\in N}\{\hat{\omega}_{A_1}(x_i)\vee\hat{\omega}_{B_1}(y_i)\}}\}\\
&\geq\inf\limits_{x=\sum\limits_{i\in N}\alpha_i[x_i,y_i]}\{\max\limits_{i\in N}\{\hat{r}_{A_2}(x_i)\vee\hat{r}_{B_2}(y_i)\}e^{i2\pi\max\limits_{i\in N}\{\hat{\omega}_{A_2}(x_i)\vee\hat{\omega}_{B_2}(y_i)\}}\}=\rho_{[A_2,B_2]}(x)
\end{align*}
\end{proof}
\begin{Theorem}\label{thrm-1}\textnormal{
Let $A_1=(\lambda_{A_1},\rho_{A_1})$, $A_2=(\lambda_{A_2},\rho_{A_2})$ and $B_1=(\lambda_{B_1},\rho_{B_1})$, $B_2=(\lambda_{B_2},\rho_{B_2})$ and $A=(\lambda_A,\rho_A)$, $B=(\lambda_B,\rho_B)$  be CIF sets of $\mathbf{V}$. Then
$[A_1+A_2,B]=[A_1,B]+[A_2,B]$ and $[A,B_1+B_2]=[A,B_1]+[A,B_2]$.}
\end{Theorem}
\begin{proof}
Let $x\in \mathbf{V}$. If we define $\overline{\sup}$ by $\overline{\sup}=\sup\limits_{x=\sum\limits_{i\in N}\alpha_i[x_i,y_i]}$. Then
\begin{align*}
\lambda_{[A_1+A_2,B]}(x)&=\sup\limits_{x=\sum\limits_{i\in N}\alpha_i[x_i,y_i]}\{\min\limits_{i\in N}\{r_{A_1+A_2}(x_i)\wedge r_B(y_i)\}e^{i2\pi\min\limits_{i\in N}\{\omega_{A_1+A_2}(x_i)\wedge\omega_B(y_i)\}}\}\\
&=\overline{\sup}\{\min\limits_{i\in N}\{\sup_{x_i=a_i+b_i}\{r_{A_1}(a_i)\wedge r_{A_2}(b_i)\}\wedge r_B(y_i)\}e^{i2\pi\min\limits_{i\in N}\{\sup\limits_{x_i=a_i+b_i}\{\omega_{A_1}(a_i)\wedge \omega_{A_2}(b_i)\}\wedge\omega_B(y_i)\}}\}\\
&=\overline{\sup}\{\min\limits_{i\in N}\{\sup_{x_i=a_i+b_i}\{r_{A_1}(a_i)\wedge r_{A_2}(b_i)\wedge r_B(y_i)\}\}e^{i2\pi\min\limits_{i\in N}\{\sup\limits_{x_i=a_i+b_i}\{\omega_{A_1}(a_i)\wedge \omega_{A_2}(b_i)\wedge\omega_B(y_i)\}\}}\}\\
&=\overline{\sup}\{\min\limits_{i\in N}\{\sup_{x_i=a_i+b_i}\{r_{A_1}(a_i)\wedge r_{A_2}(b_i)\wedge r_B(y_i)\}\}\}e^{i2\pi\overline{\sup}\{\min\limits_{i\in N}\{\sup\limits_{x_i=a_i+b_i}\{\omega_{A_1}(a_i)\wedge \omega_{A_2}(b_i)\wedge\omega_B(y_i)\}\}\}}\\
&=r_{[A_1+A_2,B]}(x)e^{i2\pi\omega_{[A_1+A_2,B]}(x)}\ \ \ \ ({\rm by}\ [10]).
\end{align*}
Since $A_1+A_2$ is homogenous with $B$, then by \cite{CZ} we have that $$r_{[A_1+A_2,B]}(x)\leq r_{[A_1,B]+[A_2,B]}(x)\ \ \ {\rm and}\ \ \
\omega_{[A_1+A_2,B]}(x)\leq \omega_{[A_1,B]+[A_2,B]}(x),$$ hence $\lambda_{[A_1+A_2,B]}(x)\leq\lambda_{[A_1,B]+[A_2,B]}(x).$\\ \\
Also if we define $\overline{\inf}$ by $\overline{\inf}=\inf\limits_{x=\sum\limits_{i\in N}\alpha_i[x_i,y_i]}$. Then
\begin{align*}
\rho_{[A_1+A_2,B]}(x)&=\inf\limits_{x=\sum\limits_{i\in N}\alpha_i[x_i,y_i]}\{\max\limits_{i\in N}\{\hat{r}_{A_1+A_2}(x_i)\vee\hat{r}_B(y_i)\}e^{i2\pi\max\limits_{i\in N}\{\hat{\omega}_{A_1+A_2}(x_i)\vee\hat{\omega}_B(y_i)\}}\}\\
&=\overline{\inf}\{\max\limits_{i\in N}\{\inf_{x_i=a_i+b_i}\{\hat{r}_{A_1}(a_i)\vee \hat{r}_{A_2}(b_i)\}\vee\hat{r}_B(y_i)\}e^{i2\pi\max\limits_{i\in N}\{\inf\limits_{x_i=a_i+b_i}\{\hat{\omega}_{A_1}(a_i)\vee \hat{\omega}_{A_2}(b_i)\}\vee\hat{\omega}_B(y_i)\}}\}\\
&=\overline{\inf}\{\max\limits_{i\in N}\{\inf_{x_i=a_i+b_i}\{\hat{r}_{A_1}(a_i)\vee \hat{r}_{A_2}(b_i)\vee\hat{r}_B(y_i)\}\}e^{i2\pi\max\limits_{i\in N}\{\inf\limits_{x_i=a_i+b_i}\{\hat{\omega}_{A_1}(a_i)\vee \hat{\omega}_{A_2}(b_i)\vee\hat{\omega}_B(y_i)\}\}}\}\\
&=\overline{\inf}\{\max\limits_{i\in N}\{\inf_{x_i=a_i+b_i}\{\hat{r}_{A_1}(a_i)\vee \hat{r}_{A_2}(b_i)\vee\hat{r}_B(y_i)\}\}\}e^{i2\pi\overline{\inf}\{\max\limits_{i\in N}\{\inf\limits_{x_i=a_i+b_i}\{\hat{\omega}_{A_1}(a_i)\vee \hat{\omega}_{A_2}(b_i)\vee\hat{\omega}_B(y_i)\}\}\}}\\
&=\hat{r}_{[A_1+A_2,B]}(x)e^{i2\pi\hat{\omega}_{[A_1+A_2,B]}(x)}\ \ \ \ ({\rm by}\ [10]).
\end{align*}
Again since $A_1+A_2$ is homogenous with $B$, then by \cite{CZ} we have that $$\hat{r}_{[A_1+A_2,B]}(x)\geq\hat{r}_{[A_1,B]+[A_2,B]}(x)\ \ \ {\rm and}\ \ \
\hat{\omega}_{[A_1+A_2,B]}(x)\geq\hat{\omega}_{[A_1,B]+[A_2,B]}(x),$$ hence $\rho_{[A_1+A_2,B]}(x)\geq\rho_{[A_1,B]+[A_2,B]}(x).$ This shows that $[A_1+A_2,B]\subseteq[A_1,B]+[A_2,B].$\\
Let $x\in\mathbf{V}$ we have
\begin{align*}
\lambda_{A_1+A_2}(x)&=r_{A_1+A_2}(x)e^{i2\pi\omega_{A_1+A_2}(x)}\\
&=\sup_{x=a+b}\{(r_{A_1}(a)\wedge r_{A_2}(b))\}e^{i2\pi\sup_{x=a+b}\{ (\omega_{A_1}(a)\wedge\omega_{A_2}(b))\}}
\end{align*}
and $r_{A_1+A_2}(x)=\sup\limits_{x=a+b}\{(r_{A_1}(a)\wedge r_{A_2}(b))\}\geq r_{A_1}(x)\wedge r_{A_2}(0)=r_{A_1}(x)$ and $\omega_{A_1+A_2}(x)=\sup\limits_{x=a+b}\{ \omega_{A_1}(a)\wedge\omega_{A_2}(b)\}\geq\omega_{A_1}(x)\wedge\omega_{A_2}(0)=\omega_{A_1}(x).$ Thus, $\lambda_{A_1+A_2}(x)=r_{A_1+A_2}(x)e^{i2\pi\omega_{A_1+A_2}(x)}\geq r_{A_1}(x)e^{i2\pi\omega_{A_1}(x)}=\lambda_{A_1}(x)$, and
\begin{align*}
\rho_{A_1+A_2}(x)&=\hat{r}_{A_1+A_2}(x)e^{i2\pi\hat{\omega}_{A_1+A_2}(x)}\\
&=\inf_{x=a+b}\{(\hat{r}_{A_1}(a)\vee\hat{r}_{A_2}(b))\}e^{i2\pi\inf_{x=a+b}\{ (\hat{\omega}_{A_1}(a)\vee\hat{\omega}_{A_2}(b))\}}
\end{align*}
and $\hat{r}_{A_1+A_2}(x)=\inf\limits_{x=a+b}\{(\hat{r}_{A_1}(a)\vee\hat{r}_{A_2}(b))\}\leq \hat{r}_{A_1}(x)\vee\hat{r}_{A_2}(0)=\hat{r}_{A_1}(x)$ and $\hat{\omega}_{A_1+A_2}(x)=\inf\limits_{x=a+b}\{ \hat{\omega}_{A_1}(a)\vee\hat{\omega}_{A_2}(b)\}\leq\hat{\omega}_{A_1}(x)\vee\hat{\omega}_{A_2}(0)=\hat{\omega}_{A_1}(x).$
Thus, $\rho_{A_1+A_2}(x)=\hat{r}_{A_1+A_2}(x)e^{i2\pi\hat{\omega}_{A_1+A_2}(x)}\leq \hat{r}_{A_1}(x)e^{i2\pi\hat{\omega}_{A_1}(x)}=\rho_{A_1}(x)$. Hence $A_1\subseteq A_1+A_2$. Similarly, $A_2\subseteq A_1+A_2$. By Lemma~\ref{lem-1}, we have $[A_1,B]\subseteq[A_1+A_2,B]$, $[A_2,B]\subseteq[A_1+A_2,B]$. So, by Lemma~\ref{lem-2}, $[A_1,B]+[A_2,B]\subseteq[A_1+A_2,B]$. Hence we have $[A_1+A_2,B]=[A_1,B]+[A_2,B]$.
\end{proof}
\begin{Theorem}\label{thrm-2}\textnormal{
Let $A=(\lambda_A,\rho_A)$, $B=(\lambda_B,\rho_B)$ be CIF subspaces of $\mathbf{V}$. Then for any $\alpha\in K$ we have $[\alpha A,B]=\alpha[A,B]$ and $[A,\alpha B]=\alpha[A,B]$.}
\end{Theorem}
\begin{proof}
Let $0\not=\alpha\in K$ and $x\in\mathbf{V}$. Then
\begin{align*}
\lambda_{[\alpha A,B]}(x)&=\sup\limits_{x=\sum\limits_{i\in N}\alpha_i[x_i,y_i]}\{\min\limits_{i\in N}\{r_{\alpha A}(x_i)\wedge r_B(y_i)\}e^{i2\pi\min\limits_{i\in N}\{\omega_{\alpha A}(x_i)\wedge\omega_B(y_i)\}}\}\\
&=\sup\limits_{x=\sum\limits_{i\in N}\alpha_i[x_i,y_i]}\{\min\limits_{i\in N}\{r_A(\alpha^{-1}x_i)\wedge r_B(y_i)\}e^{i2\pi\min\limits_{i\in N}\{\omega_A(\alpha^{-1}x_i)\wedge\omega_B(y_i)\}}\}\\
&=\sup\limits_{x=\sum\limits_{i\in N}\alpha\alpha_i[\alpha^{-1}x_i,y_i]}\{\min\limits_{i\in N}\{r_A(\alpha^{-1}x_i)\wedge r_B(y_i)\}e^{i2\pi\min\limits_{i\in N}\{\omega_A(\alpha^{-1}x_i)\wedge\omega_B(y_i)\}}\}\\
&=\lambda_{[A,B]}(\alpha^{-1}x)=\lambda_{\alpha[A,B]}(x)
\end{align*}
and
\begin{align*}
\rho_{[\alpha A,B]}(x)&=\inf\limits_{x=\sum\limits_{i\in N}\alpha_i[x_i,y_i]}\{\max\limits_{i\in N}\{\hat{r}_{\alpha A}(x_i)\vee\hat{r}_B(y_i)\}e^{i2\pi\max\limits_{i\in N}\{\hat{\omega}_{\alpha A}(x_i)\vee\hat{\omega}_B(y_i)\}}\}\\
&=\inf\limits_{x=\sum\limits_{i\in N}\alpha_i[x_i,y_i]}\{\max\limits_{i\in N}\{\hat{r}_A(\alpha^{-1}x_i)\vee\hat{r}_B(y_i)\}e^{i2\pi\max\limits_{i\in N}\{\hat{\omega}_A(\alpha^{-1}x_i)\vee\hat{\omega}_B(y_i)\}}\}\\
&=\inf\limits_{x=\sum\limits_{i\in N}\alpha\alpha_i[\alpha^{-1}x_i,y_i]}\{\max\limits_{i\in N}\{\hat{r}_A(\alpha^{-1}x_i)\vee\hat{r}_B(y_i)\}e^{i2\pi\max\limits_{i\in N}\{\hat{\omega}_A(\alpha^{-1}x_i)\vee\hat{\omega}_B(y_i)\}}\}\\
&=\rho_{[A,B]}(\alpha^{-1}x)=\rho_{\alpha[A,B]}(x).
\end{align*}
If $\alpha=0$, $x\not=0$, recall that $$\lambda_{[\alpha A,B]}(x)=\sup\limits_{x=\sum\limits_{i\in N}\alpha_i[x_i,y_i]}\{\min\limits_{i\in N}\{r_{\alpha A}(x_i)\wedge r_B(y_i)\}e^{i2\pi\min\limits_{i\in N}\{\omega_{\alpha A}(x_i)\wedge\omega_B(y_i)\}}\}$$ and
$$\rho_{[\alpha A,B]}(x)=\inf\limits_{x=\sum\limits_{i\in N}\alpha_i[x_i,y_i]}\{\max\limits_{i\in N}\{\hat{r}_{\alpha A}(x_i)\vee\hat{r}_B(y_i)\}e^{i2\pi\max\limits_{i\in N}\{\hat{\omega}_{\alpha A}(x_i)\vee\hat{\omega}_B(y_i)\}}\},$$ there exists $x_i\not=0$, which implies that $r_{\alpha A}(x_i)=0$, $\omega_{\alpha A}(x_i)=0$ and $\hat{r}_{\alpha A}(x_i)=1$, $\hat{\omega}_{\alpha A}(x_i)=1$. So, $\lambda_{[\alpha A,B]}(x)=0$, $\rho_{[\alpha A,B]}(x)=1$. If $\alpha=0$, $x=0$, it is obvious. So $[\alpha A,B]=\alpha[A,B]$. The second one can be obtained in the same way.
\end{proof}
In the following theorem we show that the complex intuitionistic fuzzy bracket product $[,]$ remains bilinear.
\begin{Theorem}\label{thrm-9}\textnormal{
Let $A_1=(\lambda_{A_1},\rho_{A_1})$, $A_2=(\lambda_{A_2},\rho_{A_2})$ and $B_1=(\lambda_{B_1},\rho_{B_1})$, $B_2=(\lambda_{B_2},\rho_{B_2})$ and $A=(\lambda_A,\rho_A)$, $B=(\lambda_B,\rho_B)$  be CIF subspaces of $\mathbf{V}$. Then for any $\alpha,\beta\in K$, we have
$[\alpha A_1+\beta A_2,B]=\alpha[A_1,B]+\beta[A_2,B]$\\
$[A,\alpha B_1+\beta B_2]=\alpha[A,B_1]+\beta[A,B_2]$.}
\end{Theorem}
\begin{proof}
The results follow from Theorem~\ref{thrm-1} and Theorem~\ref{thrm-2}.
\end{proof}
\begin{lemma}\label{lem-3}\textnormal{
Let $A=(\lambda_A,\rho_A)$ and $B=(\lambda_B,\rho_B)$  be CIF vector subspaces of $\mathbf{V}$. Then $[A,B]$ is a CIF vector subspace of $\mathbf{V}$.}
\end{lemma}
\begin{proof}
For any $x,y\in\mathbf{V}$ and $\alpha\in K$.\\
(1) Suppose that $\lambda_{[A,B]}(x+y)=r_{[A,B]}(x+y)e^{i2\pi\omega_{[A,B]}(x+y)}<\lambda_{[A,B]}(x)\wedge\lambda_{[A,B]}(y)$, since $$\lambda_{[A,B]}(x)=r_{[A,B]}(x)e^{i2\pi\omega_{[A,B]}(x)}\ \ \ {\rm and}\ \ \ \  \lambda_{[A,B]}(y)=r_{[A,B]}(y)e^{i2\pi\omega_{[A,B]}(y)},$$ we have that
\begin{align*}
r_{[A,B]}(x+y)e^{i2\pi\omega_{[A,B]}(x+y)}<&r_{[A,B]}(x)e^{i2\pi\omega_{[A,B]}(x)}\wedge r_{[A,B]}(y)e^{i2\pi\omega_{[A,B]}(y)}\\
=&(r_{[A,B]}(x)\wedge r_{[A,B]}(y))e^{i2\pi(\omega_{[A,B]}(x)\wedge\omega_{[A,B]}(y))}.
\end{align*}
So we have that $r_{[A,B]}(x+y)<r_{[A,B]}(x)\wedge r_{[A,B]}(y)$ or $\omega_{[A,B]}(x+y)<\omega_{[A,B]}(x)\wedge\omega_{[A,B]}(y)$. If $r_{[A,B]}(x+y)<r_{[A,B]}(x)\wedge r_{[A,B]}(y)$, choose a number $t\in[0,1]$ such that $r_{[A,B]}(x+y)<t<r_{[A,B]}(x)$ and $r_{[A,B]}(x+y)<t<r_{[A,B]}(y)$, then there exist $x_i,y_i,z_j,w_j\in\mathbf{V}$ such that $x=\sum_{i\in\mathbb{N}}\alpha_i[x_i,y_i]$, $y=\sum_{j\in\mathbb{N}'}\alpha_j[z_j,w_j]$ and for all $i\in\mathbb{N}$, $j\in\mathbb{N}'$, we have $r_A(x_i)>t$, $r_B(y_i)>t$, $r_A(z_j)>t$, $r_B(w_j)>t$. Since $x+y=\sum_{i\in\mathbb{N}}\alpha_i[x_i,y_i]+\sum_{j\in\mathbb{N}'}\alpha_j[z_j,w_j]$, we get $r_{[A,B]}(x+y)=r_{[A,B]}(\sum_{i\in\mathbb{N}}\alpha_i[x_i,y_i]+\sum_{j\in\mathbb{N}'}\alpha_j[z_j,w_j])>t>r_{[A,B]}(x+y)$. This is a contradiction. The other case can be proved similarly, so $\lambda_{[A,B]}(x+y)\geq\lambda_{[A,B]}(x)\wedge\lambda_{[A,B]}(y)$.

Similarly, suppose that $\rho_{[A,B]}(x+y)=\hat{r}_{[A,B]}(x+y)e^{i2\pi\hat{\omega}_{[A,B]}(x+y)}>\rho_{[A,B]}(x)\vee\rho_{[A,B]}(y)$, since $$\rho_{[A,B]}(x)=\hat{r}_{[A,B]}(x)e^{i2\pi\hat{\omega}_{[A,B]}(x)}\ \ \ {\rm and}\ \ \ \  \rho_{[A,B]}(y)=\hat{r}_{[A,B]}(y)e^{i2\pi\hat{\omega}_{[A,B]}(y)},$$ we have that
\begin{align*}
\hat{r}_{[A,B]}(x+y)e^{i2\pi\hat{\omega}_{[A,B]}(x+y)}>&\hat{r}_{[A,B]}(x)e^{i2\pi\hat{\omega}_{[A,B]}(x)}\vee \hat{r}_{[A,B]}(y)e^{i2\pi\hat{\omega}_{[A,B]}(y)}\\
=&(\hat{r}_{[A,B]}(x)\vee\hat{r}_{[A,B]}(y))e^{i2\pi(\hat{\omega}_{[A,B]}(x)\vee\hat{\omega}_{[A,B]}(y))}.
\end{align*}
So we have that $\hat{r}_{[A,B]}(x+y)>\hat{r}_{[A,B]}(x)\vee\hat{r}_{[A,B]}(y)$ or $\hat{\omega}_{[A,B]}(x+y)>\hat{\omega}_{[A,B]}(x)\vee\hat{\omega}_{[A,B]}(y)$. If $\hat{r}_{[A,B]}(x+y)>\hat{r}_{[A,B]}(x)\vee\hat{r}_{[A,B]}(y)$, choose a number $t\in[0,1]$ such that $\hat{r}_{[A,B]}(x+y)>t>\hat{r}_{[A,B]}(x)$ and $\hat{r}_{[A,B]}(x+y)>t>\hat{r}_{[A,B]}(y)$, then there exist $x_i,y_i,z_j,w_j\in\mathbf{V}$ such that $x=\sum_{i\in\mathbb{N}}\alpha_i[x_i,y_i]$, $y=\sum_{j\in\mathbb{N}'}\alpha_j[z_j,w_j]$ and for all $i\in\mathbb{N}$, $j\in\mathbb{N}'$, we have $\hat{r}_A(x_i)<t$, $\hat{r}_B(y_i)<t$, $\hat{r}_A(z_j)<t$, $\hat{r}_B(w_j)<t$. Since $x+y=\sum_{i\in\mathbb{N}}\alpha_i[x_i,y_i]+\sum_{j\in\mathbb{N}'}\alpha_j[z_j,w_j]$, we get $\hat{r}_{[A,B]}(x+y)=\hat{r}_{[A,B]}(\sum_{i\in\mathbb{N}}\alpha_i[x_i,y_i]+\sum_{j\in\mathbb{N}'}\alpha_j[z_j,w_j])<t<\hat{r}_{[A,B]}(x+y)$. This is a contradiction. The other case can be proved similarly, so $\rho_{[A,B]}(x+y)\leq\rho_{[A,B]}(x)\vee\rho_{[A,B]}(y)$.\\
(2) By using the definition of $\lambda_{[A,B]}(x)$ and $\rho_{[A,B]}(x)$, we can easily show that $\lambda_{[A,B]}(\alpha x)\geq\lambda_{[A,B]}(x)$ and $\rho_{[A,B]}(\alpha x)\leq\rho_{[A,B]}(x)$.
\end{proof}
Let $A=(\lambda_A,\rho_A)$ and $B=(\lambda_B,\rho_B)$  be $\mathbb{Z}_2$-graded CIF vector subspaces of $\mathbf{V}$. Then $A=A_0\oplus A_1$ and  $B=B_0\oplus B_1$, where $A_0$, $B_0$ are CIF vector subspaces of $\mathbf{V}_0$ and $A_1$, $B_1$ are CIF vector subspaces of $\mathbf{V}_1$. Define $[A_\alpha,B_\beta]=(\lambda_{[A_\alpha,B_\beta]},\rho_{[A_\alpha,B_\beta]}),$ where $$\lambda_{[A_\alpha,B_\beta]}=\sup\limits_{x=\sum\limits_{i\in N}\alpha_i[x_i,y_i]}\{\min\limits_{i\in N}\{r_{A_\alpha}(x_i)\wedge r_{B_\beta}(y_i)\}e^{i2\pi\min\limits_{i\in N}\{\omega_{A_\alpha}(x_i)\wedge\omega_{B_\beta}(y_i)\}}\}$$ and
$$\rho_{[A_\alpha,B_\beta]}=\inf\limits_{x=\sum\limits_{i\in N}\alpha_i[x_i,y_i]}\{\max\limits_{i\in N}\{\hat{r}_{A_\alpha}(x_i)\vee\hat{r}_{B_\beta}(y_i)\}e^{i2\pi\max\limits_{i\in N}\{\hat{\omega}_{A_\alpha}(x_i)\vee\hat{\omega}_{B_\beta}(y_i)\}}\},$$ for $x_i\in\mathbf{V}_\alpha$, $y_i\in\mathbf{V}_\beta$ and $\alpha,\beta\in\mathbb{Z}_2$.\\
Note that, by Lemma~\ref{lem-3}, $[A_0,B_0]+[A_1,B_1]$ is a CIF vector subspace of $\mathbf{V}_0$ and $[A_0,B_1]+[A_1,B_0]$ is a CIF vector subspace of $\mathbf{V}_1.$
\begin{lemma}\label{lem-4}\textnormal{
Let $A=(\lambda_A,\rho_A)$ and $B=(\lambda_B,\rho_B)$  be any two $\mathbb{Z}_2$-graded CIF vector subspaces of $\mathbf{V}$. Then $[A,B]$ is a $\mathbb{Z}_2$-graded CIF vector subspace of $\mathbf{V}$.}
\end{lemma}
\begin{proof}
We get from Lemma~\ref{lem-3} that $[A,B]_0:=[A_0,B_0]+[A_1,B_1]$ is a CIF vector subspace of $\mathbf{V}_0$ and $[A,B]_1:=[A_0,B_1]+[A_1,B_0]$ is a CIF vector subspace of $\mathbf{V}_1.$ Define $[\mathfrak{a},\mathfrak{b}]_0:=[\mathfrak{a}_0,\mathfrak{b}_0]+[\mathfrak{a}_1,\mathfrak{b}_1]$ and $[\mathfrak{a},\mathfrak{b}]_1:=[\mathfrak{a}_0,\mathfrak{b}_1]+[\mathfrak{a}_1,\mathfrak{b}_0].$ Let $x\in\mathbf{V}_0$ we have
\begin{align*}
&\lambda_{[\mathfrak{a},\mathfrak{b}]_0}(x)=\lambda_{[\mathfrak{a}_0,\mathfrak{b}_0]+[\mathfrak{a}_1,\mathfrak{b}_1]}(x)\\
&=\sup\limits_{x=a+b}\{\lambda_{[\mathfrak{a}_0,\mathfrak{b}_0]}(x)\wedge\lambda_{[\mathfrak{a}_1,\mathfrak{b}_1]}(x)\}\\
&=\sup\limits_{x=a+b}\{\sup\limits_{a=\sum\limits_{i\in N}\alpha_i[k_i,l_i]}\{\min\limits_{i\in N}\{r_{\mathfrak{a}_0}(k_i)\wedge r_{\mathfrak{b}_0}(l_i)\}e^{i2\pi\min\limits_{i\in N}\{\omega_{\mathfrak{a}_0}(k_i)\wedge\omega_{\mathfrak{b}_0}(l_i)\}}\}\\
&\wedge\sup\limits_{b=\sum\limits_{j\in N^{'}}\beta_j[m_j,n_j]}\{\min\limits_{j\in N^{'}}\{r_{\mathfrak{a}_1}(m_j)\wedge r_{\mathfrak{b}_1}(n_j)\}e^{i2\pi\min\limits_{j\in N^{'}}\{\omega_{\mathfrak{a}_1}(m_j)\wedge\omega_{\mathfrak{b}_1}(n_j)\}}\}\}\\
&=\sup\limits_{x=a+b}\{\sup\limits_{a=\sum\limits_{i\in N}\alpha_i[k_i,l_i]}\{\min\limits_{i\in N}\{r_{A_0}(k_i)\wedge r_{B_0}(l_i)\}e^{i2\pi\min\limits_{i\in N}\{\omega_{A_0}(k_i)\wedge\omega_{B_0}(l_i)\}}\}\\
&\wedge\sup\limits_{b=\sum\limits_{j\in N^{'}}\beta_j[m_j,n_j]}\{\min\limits_{j\in N^{'}}\{r_{A_1}(m_j)\wedge r_{B_1}(n_j)\}e^{i2\pi\min\limits_{j\in N^{'}}\{\omega_{A_1}(m_j)\wedge\omega_{B_1}(n_j)\}}\}\}\\
&=\sup\limits_{x=a+b}\{\lambda_{[A_0,B_0]}(a)\wedge\lambda_{[A_1,B_1]}(b)\}\\
&=\lambda_{[A_0,B_0]+[A_1,B_1]}(x)=\lambda_{[A,B]_0}(x)
\end{align*}
and
\begin{align*}
&\rho_{[\mathfrak{a},\mathfrak{b}]_0}(x)=\rho_{[\mathfrak{a}_0,\mathfrak{b}_0]+[\mathfrak{a}_1,\mathfrak{b}_1]}(x)\\
&=\inf\limits_{x=a+b}\{\rho_{[\mathfrak{a}_0,\mathfrak{b}_0]}(x)\vee\rho_{[\mathfrak{a}_1,\mathfrak{b}_1]}(x)\}\\
&=\inf\limits_{x=a+b}\{\inf\limits_{a=\sum\limits_{i\in N}\alpha_i[k_i,l_i]}\{\max\limits_{i\in N}\{\hat{r}_{\mathfrak{a}_0}(k_i)\vee\hat{r}_{\mathfrak{b}_0}(l_i)\}e^{i2\pi\max\limits_{i\in N}\{hat{\omega}_{\mathfrak{a}_0}(k_i)\vee\hat{\omega}_{\mathfrak{b}_0}(l_i)\}}\}\\
&\vee\inf\limits_{b=\sum\limits_{j\in N^{'}}\beta_j[m_j,n_j]}\{\max\limits_{j\in N^{'}}\{\hat{r}_{\mathfrak{a}_1}(m_j)\vee \hat{r}_{\mathfrak{b}_1}(n_j)\}e^{i2\pi\max\limits_{j\in N^{'}}\{\hat{\omega}_{\mathfrak{a}_1}(m_j)\vee\hat{\omega}_{\mathfrak{b}_1}(n_j)\}}\}\}\\
&=\inf\limits_{x=a+b}\{\inf\limits_{a=\sum\limits_{i\in N}\alpha_i[k_i,l_i]}\{\max\limits_{i\in N}\{\hat{r}_{A_0}(k_i)\vee\hat{r}_{B_0}(l_i)\}e^{i2\pi\max\limits_{i\in N}\{\hat{\omega}_{A_0}(k_i)\vee\hat{\omega}_{B_0}(l_i)\}}\}\\
&\vee\inf\limits_{b=\sum\limits_{j\in N^{'}}\beta_j[m_j,n_j]}\{\max\limits_{j\in N^{'}}\{\hat{r}_{A_1}(m_j)\vee \hat{r}_{B_1}(n_j)\}e^{i2\pi\max\limits_{j\in N^{'}}\{\hat{\omega}_{A_1}(m_j)\vee\hat{\omega}_{B_1}(n_j)\}}\}\}\\
&=\inf\limits_{x=a+b}\{\rho_{[A_0,B_0]}(a)\vee\rho_{[A_1,B_1]}(b)\}\\
&=\rho_{[A_0,B_0]+[A_1,B_1]}(x)=\rho_{[A,B]_0}(x)
\end{align*}
Let $x\not\in\mathbf{V}_0$, then it easy to check that $\lambda_{[\mathfrak{a},\mathfrak{b}]_0}(x)=0$ and $\rho_{[\mathfrak{a},\mathfrak{b}]_0}(x)=1$. Similarly, for $x\in\mathbf{V}_1$, we have $\lambda_{[\mathfrak{a},\mathfrak{b}]_1}(x)=\lambda_{[A,B]_1}(x)$ and $\rho_{[\mathfrak{a},\mathfrak{b}]_1}(x)=\rho_{[A,B]_1}(x)$, for $x\not\in\mathbf{V}_1$, it is also easy to check that $\lambda_{[\mathfrak{a},\mathfrak{b}]_0}(x)=0$ and $\rho_{[\mathfrak{a},\mathfrak{b}]_0}(x)=1$. Then $[\mathfrak{a},\mathfrak{b}]_0$ and $[\mathfrak{a},\mathfrak{b}]_1$ are extensions of $[A,B]_0$ and $[A,B]_1$ respectively.\\
Clearly, $[\mathfrak{a},\mathfrak{b}]_0\cap[\mathfrak{a},\mathfrak{b}]_1=(\lambda_{[\mathfrak{a},\mathfrak{b}]_0\cap[\mathfrak{a},\mathfrak{b}]_1},\rho_{[\mathfrak{a},\mathfrak{b}]_0\cap[\mathfrak{a},\mathfrak{b}]_1})$, where
$$\lambda_{[\mathfrak{a},\mathfrak{b}]_0\cap[\mathfrak{a},\mathfrak{b}]_1}(x)=\lambda_{[\mathfrak{a},\mathfrak{b}]_0}(x)\wedge\lambda_{[\mathfrak{a},\mathfrak{b}]_1}(x)=\begin{cases} 0 & :\quad x\not=0\\
1 & :\quad x=0,\end{cases}$$and
$$\rho_{[\mathfrak{a},\mathfrak{b}]_0\cap[\mathfrak{a},\mathfrak{b}]_1}(x)=\rho_{[\mathfrak{a},\mathfrak{b}]_0}(x)\vee\rho_{[\mathfrak{a},\mathfrak{b}]_1}(x)\begin{cases} 1 & :\quad x\not=0\\
0 & :\quad x=0\end{cases}.$$
Let $x\in\mathbf{V}$, then
\begin{align*}
\lambda_{[A,B]}(x)&=\lambda_{[\mathfrak{a}_0+\mathfrak{a}_1,\mathfrak{b}_0+\mathfrak{b}_1]}(x)\\
&=\lambda_{([\mathfrak{a}_0,\mathfrak{b}_0]+[\mathfrak{a}_1,\mathfrak{b}_1]+[\mathfrak{a}_0,\mathfrak{b}_1]+[\mathfrak{a}_1,\mathfrak{b}_0])}(x)\\
&=\lambda_{([\mathfrak{a},\mathfrak{b}]_0+[\mathfrak{a},\mathfrak{b}]_1)}(x),
\end{align*}
and
\begin{align*}
\rho_{[A,B]}(x)&=\rho_{[\mathfrak{a}_0+\mathfrak{a}_1,\mathfrak{b}_0+\mathfrak{b}_1]}(x)\\
&=\rho_{([\mathfrak{a}_0,\mathfrak{b}_0]+[\mathfrak{a}_1,\mathfrak{b}_1]+[\mathfrak{a}_0,\mathfrak{b}_1]+[\mathfrak{a}_1,\mathfrak{b}_0])}(x)\\
&=\rho_{([\mathfrak{a},\mathfrak{b}]_0+[\mathfrak{a},\mathfrak{b}]_1)}(x).
\end{align*}
Hence $[A,B]=[A,B]_0\oplus[A,B]_1$ is a $\mathbb{Z}_2$-graded CIF subspace of $\mathbf{V}$.
\end{proof}
\begin{lemma}\label{lem-5}\textnormal{
Let $A=(\lambda_A,\rho_A)$ and $B=(\lambda_B,\rho_B)$  be any two $\mathbb{Z}_2$-graded CIF vector subspaces of $\mathbf{V}$. Then $[A,B]=[B,A]$.}
\end{lemma}
\begin{proof}
We know, by Lemma~\ref{lem-4}, that $[A,B]$ is a $\mathbb{Z}_2$-graded CIF subspace of $\mathbf{V}$, so for any $x=x_0+x_1\in\mathbf{V}$, we have $$\lambda_{[A,B]}(x)=\lambda_{([A,B]_0\oplus[A,B]_1)}(x)=\lambda_{[A,B]_0}(x_0)\wedge\lambda_{[A,B]_1}(x_1)$$ and
$$\rho_{[A,B]}(x)=\rho_{([A,B]_0\oplus[A,B]_1)}(x)=\rho_{[A,B]_0}(x_0)\vee\rho_{[A,B]_1}(x_1).$$
Hence,
\begin{align*}
&\lambda_{[A,B]_0}(x_0)=\lambda_{([A_0,B_0]+[A_1,B_1])}(x_0)
=\sup\limits_{x_0=a+b}\{\lambda_{[A_0,B_0]}(a)\wedge\lambda_{[A_1,B_1]}(b)\}\\
&=\sup\limits_{x_0=a+b}\{\sup\limits_{a=\sum\limits_{i\in N}\alpha_i[k_i,l_i]}\{\min\limits_{i\in N}\{r_{A_0}(k_i)\wedge r_{B_0}(l_i)\}e^{i2\pi\min\limits_{i\in N}\{\omega_{A_0}(k_i)\wedge\omega_{B_0}(l_i)\}}\}\\
&\wedge\sup\limits_{b=\sum\limits_{j\in N^{'}}\beta_j[m_j,n_j]}\{\min\limits_{j\in N^{'}}\{r_{A_1}(m_j)\wedge r_{B_1}(n_j)\}e^{i2\pi\min\limits_{j\in N^{'}}\{\omega_{A_1}(m_j)\wedge\omega_{B_1}(n_j)\}}\}\}\\
&=\sup\limits_{x_0=a+b}\{\sup\limits_{-a=\sum\limits_{i\in N}\alpha_i[l_i,k_i]}\{\min\limits_{i\in N}\{r_{B_0}(l_i)\wedge r_{A_0}(k_i)\}e^{i2\pi\min\limits_{i\in N}\{\omega_{B_0}(l_i)\wedge\omega_{A_0}(k_i)\}}\}\\
&\wedge\sup\limits_{b=\sum\limits_{j\in N^{'}}\beta_j[n_j,m_j]}\{\min\limits_{j\in N^{'}}\{r_{B_1}(n_j)\wedge r_{A_1}(m_j)\}e^{i2\pi\min\limits_{j\in N^{'}}\{\omega_{B_1}(n_j)\wedge\omega_{A_1}(m_j)\}}\}\}\\
&=\sup\limits_{x_0=a+b}\{\sup\limits_{a=\sum\limits_{i\in N}(-\alpha_i)[l_i,k_i]}\{\min\limits_{i\in N}\{r_{B_0}(l_i)\wedge r_{A_0}(k_i)\}e^{i2\pi\min\limits_{i\in N}\{\omega_{B_0}(l_i)\wedge\omega_{A_0}(k_i)\}}\}\\
&\wedge\sup\limits_{b=\sum\limits_{j\in N^{'}}\beta_j[n_j,m_j]}\{\min\limits_{j\in N^{'}}\{r_{B_1}(n_j)\wedge r_{A_1}(m_j)\}e^{i2\pi\min\limits_{j\in N^{'}}\{\omega_{B_1}(n_j)\wedge\omega_{A_1}(m_j)\}}\}\}\\
&=\sup\limits_{x_0=a+b}\{\lambda_{[B_0,A_0]}(a)\wedge\lambda_{[B_1,A_1]}(b)\}=\lambda_{[B,A]_0}(x_0),
\end{align*}
\begin{align*}
&\lambda_{[A,B]_1}(x_1)=\lambda_{([A_0,B_1]+[A_1,B_0])}(x_1)
=\sup\limits_{x_1=a+b}\{\lambda_{[A_0,B_1]}(a)\wedge\lambda_{[A_1,B_0]}(b)\}\\
&=\sup\limits_{x_1=a+b}\{\sup\limits_{a=\sum\limits_{i\in N}\alpha_i[k_i,l_i]}\{\min\limits_{i\in N}\{r_{A_0}(k_i)\wedge r_{B_1}(l_i)\}e^{i2\pi\min\limits_{i\in N}\{\omega_{A_0}(k_i)\wedge\omega_{B_1}(l_i)\}}\}\\
&\wedge\sup\limits_{b=\sum\limits_{j\in N^{'}}\beta_j[m_j,n_j]}\{\min\limits_{j\in N^{'}}\{r_{A_1}(m_j)\wedge r_{B_0}(n_j)\}e^{i2\pi\min\limits_{j\in N^{'}}\{\omega_{A_1}(m_j)\wedge\omega_{B_0}(n_j)\}}\}\}\\
&=\sup\limits_{x_1=a+b}\{\sup\limits_{-a=\sum\limits_{i\in N}\alpha_i[l_i,k_i]}\{\min\limits_{i\in N}\{r_{B_1}(l_i)\wedge r_{A_0}(k_i)\}e^{i2\pi\min\limits_{i\in N}\{\omega_{B_1}(l_i)\wedge\omega_{A_0}(k_i)\}}\}\\
&\wedge\sup\limits_{b=\sum\limits_{j\in N^{'}}\beta_j[n_j,m_j]}\{\min\limits_{j\in N^{'}}\{r_{B_0}(n_j)\wedge r_{A_1}(m_j)\}e^{i2\pi\min\limits_{j\in N^{'}}\{\omega_{B_0}(n_j)\wedge\omega_{A_1}(m_j)\}}\}\}\\
&=\sup\limits_{x_1=a+b}\{\sup\limits_{a=\sum\limits_{i\in N}(-\alpha_i)[l_i,k_i]}\{\min\limits_{i\in N}\{r_{B_1}(l_i)\wedge r_{A_0}(k_i)\}e^{i2\pi\min\limits_{i\in N}\{\omega_{B_1}(l_i)\wedge\omega_{A_0}(k_i)\}}\}\\
&\wedge\sup\limits_{b=\sum\limits_{j\in N^{'}}\beta_j[n_j,m_j]}\{\min\limits_{j\in N^{'}}\{r_{B_0}(n_j)\wedge r_{A_1}(m_j)\}e^{i2\pi\min\limits_{j\in N^{'}}\{\omega_{B_0}(n_j)\wedge\omega_{A_1}(m_j)\}}\}\}\\
&=\sup\limits_{x_1=a+b}\{\lambda_{[B_1,A_0]}(a)\wedge\lambda_{[B_0,A_1]}(b)\}=\lambda_{[B,A]_1}(x_1),
\end{align*}
and
\begin{align*}
&\rho_{[A,B]_0}(x_0)=\rho_{([A_0,B_0]+[A_1,B_1])}(x_0)
=\inf\limits_{x_0=a+b}\{\rho_{[A_0,B_0]}(a)\vee\rho_{[A_1,B_1]}(b)\}\\
&=\inf\limits_{x_0=a+b}\{\inf\limits_{a=\sum\limits_{i\in N}\alpha_i[k_i,l_i]}\{\max\limits_{i\in N}\{\hat{r}_{A_0}(k_i)\vee\hat{r}_{B_0}(l_i)\}e^{i2\pi\max\limits_{i\in N}\{\hat{\omega}_{A_0}(k_i)\vee\hat{\omega}_{B_0}(l_i)\}}\}\\
&\vee\inf\limits_{b=\sum\limits_{j\in N^{'}}\beta_j[m_j,n_j]}\{\max\limits_{j\in N^{'}}\{\hat{r}_{A_1}(m_j)\vee\hat{r}_{B_1}(n_j)\}e^{i2\pi\max\limits_{j\in N^{'}}\{\hat{\omega}_{A_1}(m_j)\vee\hat{\omega}_{B_1}(n_j)\}}\}\}\\
&=\inf\limits_{x_0=a+b}\{\inf\limits_{-a=\sum\limits_{i\in N}\alpha_i[l_i,k_i]}\{\max\limits_{i\in N}\{\hat{r}_{B_0}(l_i)\vee\hat{r}_{A_0}(k_i)\}e^{i2\pi\max\limits_{i\in N}\{\hat{\omega}_{B_0}(l_i)\vee\hat{\omega}_{A_0}(k_i)\}}\}\\
&\vee\inf\limits_{b=\sum\limits_{j\in N^{'}}\beta_j[n_j,m_j]}\{\max\limits_{j\in N^{'}}\{\hat{r}_{B_1}(n_j)\vee\hat{r}_{A_1}(m_j)\}e^{i2\pi\max\limits_{j\in N^{'}}\{\hat{\omega}_{B_1}(n_j)\vee\hat{\omega}_{A_1}(m_j)\}}\}\}\\
&=\inf\limits_{x_0=a+b}\{\inf\limits_{a=\sum\limits_{i\in N}(-\alpha_i)[l_i,k_i]}\{\max\limits_{i\in N}\{\hat{r}_{B_0}(l_i)\vee \hat{r}_{A_0}(k_i)\}e^{i2\pi\max\limits_{i\in N}\{\hat{\omega}_{B_0}(l_i)\vee\hat{\omega}_{A_0}(k_i)\}}\}\\
&\vee\inf\limits_{b=\sum\limits_{j\in N^{'}}\beta_j[n_j,m_j]}\{\max\limits_{j\in N^{'}}\{\hat{r}_{B_1}(n_j)\vee\hat{r}_{A_1}(m_j)\}e^{i2\pi\max\limits_{j\in N^{'}}\{\hat{\omega}_{B_1}(n_j)\vee\hat{\omega}_{A_1}(m_j)\}}\}\}\\
&=\inf\limits_{x_0=a+b}\{\rho_{[B_0,A_0]}(a)\vee\rho_{[B_1,A_1]}(b)\}=\rho_{[B,A]_0}(x_0),
\end{align*}
\begin{align*}
&\rho_{[A,B]_1}(x_1)=\rho_{([A_0,B_1]+[A_1,B_0])}(x_1)
=\inf\limits_{x_1=a+b}\{\rho_{[A_0,B_1]}(a)\vee\rho_{[A_1,B_0]}(b)\}\\
&=\inf\limits_{x_1=a+b}\{\inf\limits_{a=\sum\limits_{i\in N}\alpha_i[k_i,l_i]}\{\max\limits_{i\in N}\{\hat{r}_{A_0}(k_i)\vee\hat{r}_{B_1}(l_i)\}e^{i2\pi\max\limits_{i\in N}\{\hat{\omega}_{A_0}(k_i)\vee\hat{\omega}_{B_1}(l_i)\}}\}\\
&\vee\inf\limits_{b=\sum\limits_{j\in N^{'}}\beta_j[m_j,n_j]}\{\max\limits_{j\in N^{'}}\{\hat{r}_{A_1}(m_j)\vee\hat{r}_{B_0}(n_j)\}e^{i2\pi\max\limits_{j\in N^{'}}\{\hat{\omega}_{A_1}(m_j)\vee\hat{\omega}_{B_0}(n_j)\}}\}\}\\
&=\inf\limits_{x_1=a+b}\{\inf\limits_{-a=\sum\limits_{i\in N}\alpha_i[l_i,k_i]}\{\max\limits_{i\in N}\{\hat{r}_{B_1}(l_i)\vee\hat{r}_{A_0}(k_i)\}e^{i2\pi\max\limits_{i\in N}\{\hat{\omega}_{B_1}(l_i)\vee\hat{\omega}_{A_0}(k_i)\}}\}\\
&\vee\inf\limits_{b=\sum\limits_{j\in N^{'}}\beta_j[n_j,m_j]}\{\max\limits_{j\in N^{'}}\{\hat{r}_{B_0}(n_j)\vee\hat{r}_{A_1}(m_j)\}e^{i2\pi\max\limits_{j\in N^{'}}\{\hat{\omega}_{B_0}(n_j)\vee\hat{\omega}_{A_1}(m_j)\}}\}\}\\
&=\inf\limits_{x_1=a+b}\{\inf\limits_{a=\sum\limits_{i\in N}(-\alpha_i)[l_i,k_i]}\{\max\limits_{i\in N}\{\hat{r}_{B_1}(l_i)\vee \hat{r}_{A_0}(k_i)\}e^{i2\pi\max\limits_{i\in N}\{\hat{\omega}_{B_1}(l_i)\vee\hat{\omega}_{A_0}(k_i)\}}\}\\
&\vee\inf\limits_{b=\sum\limits_{j\in N^{'}}\beta_j[n_j,m_j]}\{\max\limits_{j\in N^{'}}\{\hat{r}_{B_0}(n_j)\vee\hat{r}_{A_1}(m_j)\}e^{i2\pi\max\limits_{j\in N^{'}}\{\hat{\omega}_{B_0}(n_j)\vee\hat{\omega}_{A_1}(m_j)\}}\}\}\\
&=\inf\limits_{x_1=a+b}\{\rho_{[B_1,A_0]}(a)\vee\rho_{[B_0,A_1]}(b)\}=\rho_{[B,A]_1}(x_1).
\end{align*}
So $[A,B]=[B,A]$.
\end{proof}
\begin{Theorem}\label{thrm-3}\textnormal{
Let $A=(\lambda_A,\rho_A)$ and $B=(\lambda_B,\rho_B)$  be any two CIF ideals of $V$. Then $[A,B]$ is a CIF ideal of $\mathbf{V}$.}
\end{Theorem}
\begin{proof}
By Lemma~\ref{lem-5}, $[A,B]$ is a $\mathbb{Z}_2$-graded CIF subspace of $\mathbf{V}$, since $A=(\lambda_A,\rho_A)$, $B=(\lambda_B,\rho_B)$ are CIF ideals of $\mathbf{V}$. To complete the proof we must show that $\lambda_{[A,B]}([x,y])\geq\lambda_{[A,B]}(x)\vee\lambda_{[A,B]}(y)$ and $\rho_{[A,B]}([x,y])\leq\rho_{[A,B]}(x)\wedge\rho_{[A,B]}(y)$.\\
Suppose that $\lambda_{[A,B]}([x,y])<\lambda_{[A,B]}(x)\vee\lambda_{[A,B]}(y)$, then we have $\lambda_{[A,B]}([x,y])<\lambda_{[A,B]}(x)$ or $\lambda_{[A,B]}([x,y])<\lambda_{[A,B]}(y)$. If $\lambda_{[A,B]}([x,y])<\lambda_{[A,B]}(x)$, then $r_{[A,B]}([x,y])e^{i2\pi \omega_{[A,B]}([x,y])}<r_{[A,B]}(x)e^{i2\pi \omega_{[A,B]}(x)}$ which implies that $r_{[A,B]}([x,y])<r_{[A,B]}(x)$ or $\omega_{[A,B]}([x,y])<\omega_{[A,B]}(x)$. Let $r_{[A,B]}([x,y])<r_{[A,B]}(x)$. Choose a number $t\in[0,1]$ such that $r_{[A,B]}([x,y])<t<r_{[A,B]}(x)$, then there exist $x_i,y_i\in\mathbf{V}$ and $\alpha_i\in K$ such that $x=\sum_{i\in N}\alpha_i[x_i,y_i]$ and for all $i\in N$, $r_A(x_i)>t$, $r_B(y_i)>t$. Moreover, $r_A(x_i)=r_{A_0+A_1}(x_{i_0}+x_{i_1})=r_{A_0}(x_{i_0})\wedge r_{A_1}(x_{i_1})>t$, and $r_B(y_i)=r_{B_0+B_1}(y_{i_0}+y_{i_1})=r_{B_0}(y_{i_0})\wedge r_{B_1}(y_{i_1})>t$, then we have $r_{A_0}(x_{i_0})>t$, $r_{A_1}(x_{i_1})>t$ and $r_{B_0}(y_{i_0})>t$, $r_{B_1}(y_{i_1})>t$. Since $[x,y]=[\sum_{i\in N}\alpha_i[x_i,y_i],y]=\sum_{i\in N}\alpha_i[[x_i,y_i],y]$, and since
$$[[x_i,y_i],y]=[x_i,[y_i,y]]-[y_{i_0},[x_{i_0},y]]+[y_{i_1},[x_{i_1},y]]-[y_{i_1},[x_{i_0},y]]-[y_{i_0},[x_{i_1},y]]$$
we get
\begin{align*}
r_{[A,B]}([x,y])&=r_{[A,B]}(\sum_{i\in N}\alpha_i[[x_i,y_i],y])\geq r_{[A,B]}([[x_i,y_i],y])\geq\min\{r_{[A,B]}([x_i,[y_i,y]]),\\
&r_{[A,B]}(-[y_{i_0},[x_{i_0},y]]),r_{[A,B]}([y_{i_1},[x_{i_1},y]]),r_{[A,B]}(-[y_{i_1},[x_{i_0},y]]),r_{[A,B]}(-[y_{i_0},[x_{i_1},y]])\}
\end{align*}
if $r_{[A,B]}([x_i,[y_i,y]])$ is the minimum, then we have
\begin{align*}
r_{[A,B]}([x_i,[y_i,y]])&\geq r_A(x_i)\wedge r_B([y_i,y])\\
&\geq r_A(x_i)\wedge(r_B(y_i)\vee r_B(y))>t;
\end{align*}
if $r_{[A,B]}(-[y_{i_0},[x_{i_0},y]])$ is the minimum, then, by Lemma~\ref{lem-5}, we have
\begin{align*}
r_{[A,B]}(-[y_{i_0},[x_{i_0},y]])&=r_{[A,B]}([y_{i_0},[x_{i_0},y]])=r_{[B,A]}([y_{i_0},[x_{i_0},y]])\\
&\geq r_B(y_{i_0})\wedge r_A([x_{i_0},y])\\
&\geq r_B(y_{i_0})\wedge(r_A(x_{i_0})\vee r_A(y))\\
&=r_{B_0}(y_{i_0})\wedge(r_{A_0}(x_{i_0})\vee r_A(y))>t;
\end{align*}
if $r_{[A,B]}([y_{i_1},[x_{i_1},y]])$ is the minimum, then, by Lemma~\ref{lem-5}, we have
\begin{align*}
r_{[A,B]}([y_{i_1},[x_{i_1},y]])&=r_{[B,A]}([y_{i_1},[x_{i_1},y]])\\
&\geq r_B(y_{i_1})\wedge r_A([x_{i_1},y])\\
&\geq r_B(y_{i_1})\wedge(r_A(x_{i_1})\vee r_A(y))\\
&=r_{B_1}(y_{i_1})\wedge(r_{A_1}(x_{i_1})\vee r_A(y))>t;
\end{align*}
if $r_{[A,B]}(-[y_{i_1},[x_{i_0},y]])$ is the minimum, then, by Lemma~\ref{lem-5}, we have
\begin{align*}
r_{[A,B]}(-[y_{i_1},[x_{i_0},y]]&=r_{[A,B]}([y_{i_1},[x_{i_0},y]])\\
&=r_{[B,A]}([y_{i_1},[x_{i_0},y]])\\
&\geq r_B(y_{i_1})\wedge r_A([x_{i_0},y])\\
&\geq r_B(y_{i_1})\wedge(r_A(x_{i_0})\vee r_A(y))\\
&=r_{B_1}(y_{i_1})\wedge(r_{A_0}(x_{i_0})\vee r_A(y))>t;
\end{align*}
if $r_{[A,B]}(-[y_{i_0},[x_{i_1},y]])$ is the minimum, then, by Lemma~\ref{lem-5}, we have
\begin{align*}
r_{[A,B]}(-[y_{i_0},[x_{i_1},y]]&=r_{[A,B]}([y_{i_0},[x_{i_1},y]])\\
&=r_{[B,A]}([y_{i_0},[x_{i_1},y]])\\
&\geq r_B(y_{i_0})\wedge r_A([x_{i_1},y])\\
&\geq r_B(y_{i_0})\wedge(r_A(x_{i_1})\vee r_A(y))\\
&=r_{B_0}(y_{i_0})\wedge(r_{A_1}(x_{i_1})\vee r_A(y))>t.
\end{align*}
So we have $r_{[A,B]}([x,y])>t>r_{[A,B]}([x,y])$, this is a contradiction, similarly we can prove the case of $\omega_{[A,B]}([x,y])<\omega_{[A,B]}(x)$. We use the similar method of $\lambda_{[A,B]}([x,y])<\lambda_{[A,B]}(y)$.\\
Also suppose that $\rho_{[A,B]}([x,y])>\rho_{[A,B]}(x)\wedge\rho_{[A,B]}(y)$, then we have $\rho_{[A,B]}([x,y])>\rho_{[A,B]}(x)$ or $\rho_{[A,B]}([x,y])>\rho_{[A,B]}(y)$. If $\rho_{[A,B]}([x,y])>\rho_{[A,B]}(x)$, then $\hat{r}_{[A,B]}([x,y])e^{i2\pi \hat{\omega}_{[A,B]}([x,y])}>\hat{r}_{[A,B]}(x)e^{i2\pi\hat{\omega}_{[A,B]}(x)}$ which implies that $\hat{r}_{[A,B]}([x,y])>\hat{r}_{[A,B]}(x)$ or $\hat{\omega}_{[A,B]}([x,y])>\hat{\omega}_{[A,B]}(x)$. Let $\hat{r}_{[A,B]}([x,y])>\hat{r}_{[A,B]}(x)$. Then $\hat{r}_{[A,B]}([x,y])>t>\hat{r}_{[A,B]}(x)$ for some $t\in[0,1]$. So there exist $x_i,y_i\in\mathbf{V}$ and $\alpha_i\in K$ such that $x=\sum_{i\in N}\alpha_i[x_i,y_i]$ and for all $i\in N$, $\hat{r}_A(x_i)<t$, $\hat{r}_B(y_i)<t$. Moreover, $\hat{r}_A(x_i)=\hat{r}_{A_0+A_1}(x_{i_0}+x_{i_1})=\hat{r}_{A_0}(x_{i_0})\vee\hat{r}_{A_1}(x_{i_1})<t$, and $\hat{r}_B(y_i)=\hat{r}_{B_0+B_1}(y_{i_0}+y_{i_1})=\hat{r}_{B_0}(y_{i_0})\vee \hat{r}_{B_1}(y_{i_1})<t$, then we have $\hat{r}_{A_0}(x_{i_0})<t$, $\hat{r}_{A_1}(x_{i_1})<t$ and $\hat{r}_{B_0}(y_{i_0})<t$, $\hat{r}_{B_1}(y_{i_1})<t$.
We get
\begin{align*}
\hat{r}_{[A,B]}([x,y])&=\hat{r}_{[A,B]}(\sum_{i\in N}\alpha_i[[x_i,y_i],y])\leq \hat{r}_{[A,B]}([[x_i,y_i],y])\leq\max\{\hat{r}_{[A,B]}([x_i,[y_i,y]]),\\
&\hat{r}_{[A,B]}(-[y_{i_0},[x_{i_0},y]]),\hat{r}_{[A,B]}([y_{i_1},[x_{i_1},y]]), \hat{r}_{[A,B]}(-[y_{i_1},[x_{i_0},y]]),\hat{r}_{[A,B]}(-[y_{i_0},[x_{i_1},y]])\}
\end{align*}
if $\hat{r}_{[A,B]}([x_i,[y_i,y]])$ is the maximum, then we have
\begin{align*}
\hat{r}_{[A,B]}([x_i,[y_i,y]])&\leq\hat{r}_A(x_i)\vee\hat{r}_B([y_i,y])\\
&\leq\hat{r}_A(x_i)\vee(\hat{r}_B(y_i)\wedge\hat{r}_B(y))<t;
\end{align*}
if $\hat{r}_{[A,B]}(-[y_{i_0},[x_{i_0},y]])$ is the maximum, then, by Lemma~\ref{lem-5}, we have
\begin{align*}
\hat{r}_{[A,B]}(-[y_{i_0},[x_{i_0},y]])&=\hat{r}_{[A,B]}([y_{i_0},[x_{i_0},y]])=\hat{r}_{[B,A]}([y_{i_0},[x_{i_0},y]])\\
&\leq\hat{r}_B(y_{i_0})\vee\hat{r}_A([x_{i_0},y])\\
&\leq\hat{r}_B(y_{i_0})\vee(r_A(x_{i_0})\wedge\hat{r}_A(y))\\
&=\hat{r}_{B_0}(y_{i_0})\vee(\hat{r}_{A_0}(x_{i_0})\wedge\hat{r}_A(y))<t;
\end{align*}
if $\hat{r}_{[A,B]}([y_{i_1},[x_{i_1},y]])$ is the maximum, then, by Lemma~\ref{lem-5}, we have
\begin{align*}
\hat{r}_{[A,B]}([y_{i_1},[x_{i_1},y]])&=\hat{r}_{[B,A]}([y_{i_1},[x_{i_1},y]])\\
&\leq\hat{r}_B(y_{i_1})\vee\hat{r}_A([x_{i_1},y])\\
&\leq\hat{r}_B(y_{i_1})\vee(\hat{r}_A(x_{i_1})\wedge\hat{r}_A(y))\\
&=\hat{r}_{B_1}(y_{i_1})\vee(\hat{r}_{A_1}(x_{i_1})\wedge\hat{r}_A(y))<t;
\end{align*}
if $\hat{r}_{[A,B]}(-[y_{i_1},[x_{i_0},y]])$ is the maximum, then, by Lemma~\ref{lem-5}, we have
\begin{align*}
\hat{r}_{[A,B]}(-[y_{i_1},[x_{i_0},y]]&=\hat{r}_{[A,B]}([y_{i_1},[x_{i_0},y]])\\
&=\hat{r}_{[B,A]}([y_{i_1},[x_{i_0},y]])\\
&\leq\hat{r}_B(y_{i_1})\vee\hat{r}_A([x_{i_0},y])\\
&\leq\hat{r}_B(y_{i_1})\vee(\hat{r}_A(x_{i_0})\wedge\hat{r}_A(y))\\
&=\hat{r}_{B_1}(y_{i_1})\vee(\hat{r}_{A_0}(x_{i_0})\wedge\hat{r}_A(y))<t;
\end{align*}
if $\hat{r}_{[A,B]}(-[y_{i_0},[x_{i_1},y]])$ is the maximum, then, by Lemma~\ref{lem-5}, we have
\begin{align*}
\hat{r}_{[A,B]}(-[y_{i_0},[x_{i_1},y]]&=\hat{r}_{[A,B]}([y_{i_0},[x_{i_1},y]])\\
&=\hat{r}_{[B,A]}([y_{i_0},[x_{i_1},y]])\\
&\leq\hat{r}_B(y_{i_0})\vee\hat{r}_A([x_{i_1},y])\\
&\leq\hat{r}_B(y_{i_0})\vee(\hat{r}_A(x_{i_1})\wedge\hat{r}_A(y))\\
&=\hat{r}_{B_0}(y_{i_0})\vee(\hat{r}_{A_1}(x_{i_1})\wedge\hat{r}_A(y))<t.
\end{align*}
So we have $\hat{r}_{[A,B]}([x,y])<t<\hat{r}_{[A,B]}([x,y])$, this is a contradiction, similarly we can prove the case of $\hat{\omega}_{[A,B]}([x,y])>\hat{\omega}_{[A,B]}(x)$. We use the similar method of $\rho_{[A,B]}([x,y])>\rho_{[A,B]}(y)$. Hence $[A,B]$ is CIF ideal of $\mathbf{V}$.
\end{proof}
\section{Anti-homomorphisms of complex intuitionistic fuzzy brackets}
\begin{defn}\cite{AMR}
If $\phi:\mathbf{V}\rightarrow\mathbf{V}'$ is a linear map between lie superalgebras $\mathbf{V},\ \mathbf{V}'$ which satisfies:
\begin{eqnarray}
\phi(\mathbf{V}_\alpha)&\subseteq& \mathbf{V}'_\alpha ,\  (\alpha= 0,1),\\
\phi([x, y])&=& -[\phi(x), \phi(y)]
\end{eqnarray}
Then $\phi$ is called an anti-homomorphism of lie-superalgebras.
\end{defn}
\begin{defn}\cite{AMR}
Let $\mathbf{V}$, $\mathbf{V}'$ be $K$-vector spaces and let $f :\mathbf{V}\rightarrow\mathbf{V}'$ be any map. If $A=(\lambda_A, \rho_A)$, $B=(\lambda_B, \rho_B)$ are CIF vector subspaces of $\mathbf{V}$ and $\mathbf{V}'$, respectively, then the preimage of $B$ under $f$ is defined to be a CIF set $f^{-1}(B)=(\lambda_{f^{-1}(B)}, \rho_{f^{-1}(B)})$, where $\lambda_{f^{-1}(B)}(x)=\lambda_B(f(x))$ and $\rho_{f^{-1}(B)}(x)=\rho_B(f(x))$ for any $x\in\mathbf{V}$ and the image of $A=(\lambda_A, \rho_A)$ under $f$ is defined to be the CIF set $f(A)=(\lambda_{f(A)}, \rho_{f(A)})$ where $$\lambda_{f(A)}(y)=\begin{cases} \underset{x\in f^{-1}(y)}{\sup}\{\lambda_A(x)\}=\underset{x\in f^{-1}(y)}{\sup}\{r_A(x)e^{i2\pi\omega_A(x)}\} & :\quad y\in f(\mathbf{V})\\
0 & :\quad y\not\in f(\mathbf{V})\end{cases}$$ and
$$\rho_{f(A)}(y)=\begin{cases} \underset{x\in f^{-1}(y)}{\inf}\{\rho_A(x)\}=\underset{x\in f^{-1}(y)}{\inf}\{\hat{r}_A(x)e^{i2\pi\hat{\omega}_A(x)}\} & :\quad y\in f(\mathbf{V})\\
1 & :\quad y\not\in f(\mathbf{V})\end{cases}$$
\end{defn}
\begin{Theorem}\label{thrm-4}\textnormal{
Let $A=(\lambda_A,\rho_A)$, $B=(\lambda_B,\rho_B)$  be any two CIF ideals of $\mathbf{V}$, and
let $\phi:\mathbf{V}\rightarrow\mathbf{V}'$ be a surjective anti-homomorphism of lie-superalgebras. Then the CIF set $\phi([A,B])$ of $\mathbf{V}'$ is an anti-CIF ideal of $\mathbf{V}'$ and $\phi([A,B])\subseteq[\phi(A),\phi(B)]$.}
\end{Theorem}
\begin{proof}
 By Theorem~\ref{thrm-3}, $[A,B]$ is a CIF ideal of $\mathbf{V}$ and by \cite[Proposition 4.4]{AMR}, $\phi([A,B])$ is an anti-CIF ideal of $\mathbf{V}'$. To complete the proof we need to show that $[\phi(A),\phi(B)]\subseteq\phi([A,B])$. Let $y\in\mathbf{V}'$, then
\begin{align*}
\lambda_{\phi([A,B])}(y)&=\sup\limits_{y=\phi(x)}\{\lambda_{[A,B]}(x)\}\\
&=\sup\limits_{y=\phi(x)}\{\sup\limits_{x=\sum\limits_{i\in N}\alpha_i[x_i,y_i]}\{\min\limits_{i\in N}\{\lambda_A(x_i)\wedge \lambda_B(y_i)\}\}\}\\
&=\sup\limits_{y=\sum\limits_{i\in N}(-\alpha_i)[\phi(x_i),\phi(y_i)]}\{\{\min\limits_{i\in N}\{\lambda_A(x_i)\wedge \lambda_B(y_i)\}\}\},
\end{align*}
for each $i\in N$, let $m_i=\phi(x_i)$ and $n_i=\phi(y_i)$. Then $\lambda_{\phi(A)}(m_i)=\sup_{m_i=\phi(x)}\{\lambda_A(x)\}\geq\lambda_A(x_i)$, similarly, $\lambda_{\phi(A)}(n_i)\geq\lambda_A(y_i)$. So
\begin{align*}
\lambda_{\phi([A,B])}(y)&=\sup\limits_{y=\sum\limits_{i\in N}(-\alpha_i)[\phi(x_i),\phi(y_i)]}\{\{\min\limits_{i\in N}\{\lambda_A(x_i)\wedge\lambda_B(y_i)\}\}\}\\
&\leq\sup\limits_{y=\sum\limits_{i\in N}(-\alpha_i)[m_i,n_i]}\{\min\limits_{i\in N}\{\lambda_{\phi(A)}(m_i)\wedge\lambda_{\phi(B)}(n_i)\}\}\\
&=\lambda_{[\phi(A),\phi(B)]}(y).
\end{align*}
and
\begin{align*}
\rho_{\phi([A,B])}(y)&=\inf\limits_{y=\phi(x)}\{\rho_{[A,B]}(x)\}\\
&=\inf\limits_{y=\phi(x)}\{\inf\limits_{x=\sum\limits_{i\in N}\alpha_i[x_i,y_i]}\{\max\limits_{i\in N}\{\rho_A(x_i)\vee \rho_B(y_i)\}\}\}\\
&=\inf\limits_{y=\sum\limits_{i\in N}(-\alpha_i)[\phi(x_i),\phi(y_i)]}\{\{\max\limits_{i\in N}\{\rho_A(x_i)\vee \rho_B(y_i)\}\}\},
\end{align*}
for each $i\in N$, let $m_i=\phi(x_i)$ and $n_i=\phi(y_i)$. Then $\rho_{\phi(A)}(m_i)=\inf_{m_i=\phi(x)}\{\rho_A(x)\}\leq\rho_A(x_i)$, similarly, $\rho_{\phi(A)}(n_i)\leq\rho_A(y_i)$. So
\begin{align*}
\rho_{\phi([A,B])}(y)&=\inf\limits_{y=\sum\limits_{i\in N}(-\alpha_i)[\phi(x_i),\phi(y_i)]}\{\{\max\limits_{i\in N}\{\rho_A(x_i)\vee\rho_B(y_i)\}\}\}\\
&\geq\inf\limits_{y=\sum\limits_{i\in N}(-\alpha_i)[m_i,n_i]}\{\max\limits_{i\in N}\{\rho_{\phi(A)}(m_i)\vee\rho_{\phi(B)}(n_i)\}\}\\
&=\rho_{[\phi(A),\phi(B)]}(y).
\end{align*}
Hence $\phi([A,B])\subseteq[\phi(A),\phi(B)]$.
\end{proof}
\begin{Theorem}\textnormal{
Let $\phi:\mathbf{V}\rightarrow\mathbf{V}'$ be a surjective anti-homomorphism of lie-superalgebras, and Let $A=(\lambda_A,\rho_A)$, $B=(\lambda_B,\rho_B)$  be any two CIF ideals of $\mathbf{V}'$. Then the CIF set $\phi^{-1}([A,B])$ of $\mathbf{V}$ is an anti-CIF ideal of $\mathbf{V}$ and $\phi^{-1}([A,B])\subseteq[\phi^{-1}(A),\phi^{-1}(B)]$.}
\end{Theorem}
\begin{proof}
By Theorem~\ref{thrm-3}, $[A,B]$ is a CIF ideal of $\mathbf{V}'$ and by \cite[Proposition 4.3]{AMR}, $\phi^{-1}([A,B])$ is an anti-CIF ideal of $\mathbf{V}$. The rest of the proof is easy, and it will be left to the reader.
\end{proof}
\begin{Theorem}\label{thrm-15}\textnormal{
Let $\phi:\mathbf{V}\rightarrow\mathbf{V}'$ be a surjective anti-homomorphism of lie-superalgebras, and Let $A=(\lambda_A,\rho_A)$, $B=(\lambda_B,\rho_B)$  be any CIF ideals of $\mathbf{V}'$. Then the CIF set $\phi^{-1}(A+B)=\phi^{-1}(A)+\phi^{-1}(B)$ of $V$ is an anti-CIF ideal of $\mathbf{V}$.}
\end{Theorem}
\begin{proof}
Since $A$, $B$ are CIF ideals of $\mathbf{V}'$, then by \cite[Theorem 3.11]{AMR}, $A+B$ is a CIF ideal of $\mathbf{V}'$ and by \cite[Proposition 4.3]{AMR}, $\phi^{-1}(A+B)$ is an anti- CIF ideal of $\mathbf{V}$. In order to complete the proof it remains to show that $\phi^{-1}(A+B)=\phi^{-1}(A)+\phi^{-1}(B)$. Again by using the same approach used in the proof of \cite[Theorem 4.5]{AMR}, one can easily see that $\phi^{-1}(A+B)=\phi^{-1}(A)+\phi^{-1}(B)$.
\end{proof}
\begin{Theorem}\label{thrm-11}\textnormal{
Let $\phi:\mathbf{V}\rightarrow\mathbf{V}'$ be a surjective anti-homomorphism of lie-superalgebras, and Let $B=(\lambda_B,\rho_B)$  be any CIF ideal of $\mathbf{V}'$. Then, for any $\alpha\in K$, the CIF set $\phi^{-1}(\alpha B)=\alpha\phi^{-1}(B)$ of $V$ is an anti-CIF ideal of $\mathbf{V}$.}
\end{Theorem}
\begin{proof}
We already proved in \cite[Proposition 4.3]{AMR} that $\phi^{-1}(\alpha B)$ is an anti-CIF ideal of $\mathbf{V}$. Therefore the only thing we need to prove is that  $\phi^{-1}(\alpha B)=\alpha\phi^{-1}(B)$. Let $x\in \mathbf{V}$, then
\begin{align*}
\lambda_{\phi^{-1}(\alpha B)}(x)&=\lambda_{\alpha B}(\phi(x))=\lambda_B(\alpha^{-1}\phi(x))\\
&=\lambda_B(\phi(\alpha^{-1}x))=\lambda_{\phi^{-1}(B)}(\alpha^{-1}x)\\
&=\lambda_{\alpha\phi^{-1}(B)}(x)
\end{align*}
and
\begin{align*}
\rho_{\phi^{-1}(\alpha B)}(x)&=\rho_{\alpha B}(\phi(x))=\rho_B(\alpha^{-1}\phi(x))\\
&=\rho_B(\phi(\alpha^{-1}x))=\rho_{\phi^{-1}(B)}(\alpha^{-1}x)\\
&=\rho_{\alpha\phi^{-1}(B)}(x).
\end{align*}
Hence, $\phi^{-1}(\alpha B)=\alpha\phi^{-1}(B)$.
\end{proof}
\begin{Theorem}\label{thrm-10}\textnormal{
Let $\phi:\mathbf{V}\rightarrow\mathbf{V}'$ be a surjective anti-homomorphism of lie-superalgebras, and Let $A=(\lambda_A,\rho_A)$  be any CIF ideal of $\mathbf{V}$. Then, for any $\alpha\in K$, the CIF set $\phi(\alpha A)=\alpha\phi(A)$ of $\mathbf{V}'$ is an anti-CIF ideal of $\mathbf{V}'$.}
\end{Theorem}
\begin{proof}
We already proved in \cite[Proposition 4.4]{AMR} that $\phi(\alpha A)$ is an anti-CIF ideal of $\mathbf{V}'$. Therefore the only thing we need to prove is that  $\phi(\alpha A)=\alpha\phi(A)$. Let $y\in\mathbf{V}'$, then
\begin{align*}
\lambda_{\phi(\alpha A)}(y)&=\sup\limits_{x\in\phi^{-1}(y)}\{\lambda_{\alpha A}(x)\}=\sup\limits_{x\in\phi^{-1}(y)}\{\lambda_A(\alpha^{-1}(x))\}\\
&=\sup\limits_{\alpha^{-1}x\in\phi^{-1}(\alpha^{-1}y)}\{\lambda_A(\alpha^{-1}(x))\}\\
&=\lambda_{\phi(A)}(\alpha^{-1}y)=\lambda_{\alpha\phi(A)}(y)
\end{align*}
and
\begin{align*}
\rho_{\phi(\alpha A)}(y)&=\inf\limits_{x\in\phi^{-1}(y)}\{\rho_{\alpha A}(x)\}=\inf\limits_{x\in\phi^{-1}(y)}\{\rho_A(\alpha^{-1}(x))\}\\
&=\inf\limits_{\alpha^{-1}x\in\phi^{-1}(\alpha^{-1}y)}\{\rho_A(\alpha^{-1}(x))\}\\
&=\rho_{\phi(A)}(\alpha^{-1}y)=\rho_{\alpha\phi(A)}(y).
\end{align*}
Hence, $\phi(\alpha A)=\alpha\phi(A)$.
\end{proof}
\begin{Theorem}\textnormal{
Let $\phi:\mathbf{V}\rightarrow\mathbf{V}'$ be a surjective anti-homomorphism of lie-superalgebras, and
let $A_1=(\lambda_{A_1},\rho_{A_1})$, $A_2=(\lambda_{A_2},\rho_{A_2})$ and $B_1=(\lambda_{B_1},\rho_{B_1})$, $B_2=(\lambda_{B_2},\rho_{B_2})$ and $A=(\lambda_A,\rho_A)$, $B=(\lambda_B,\rho_B)$  be CIF subspaces of $\mathbf{V}$. Then for any $\alpha,\beta\in K$, we have
$[\phi(\alpha A_1+\beta A_2),\phi(B)]=\alpha[\phi(A_1),\phi(B)]+\beta[\phi(A_2),\phi(B)],$\\
$[\phi(A),\phi(\alpha B_1+\beta B_2)]=\alpha[\phi(A),\phi(B_1)]+\beta[\phi(A),\phi(B_2)]$.}
\end{Theorem}
\begin{proof}
The results follow from Theorem~\ref{thrm-9}, Theorem~\ref{thrm-10} and \cite[Theorem 4.5]{AMR}.
\end{proof}
\begin{Theorem}\textnormal{
Let $\phi:\mathbf{V}\rightarrow\mathbf{V}'$ be a surjective anti-homomorphism of lie-superalgebras, and
let $A_1=(\lambda_{A_1},\rho_{A_1})$, $A_2=(\lambda_{A_2},\rho_{A_2})$ and $B_1=(\lambda_{B_1},\rho_{B_1})$, $B_2=(\lambda_{B_2},\rho_{B_2})$ and $A=(\lambda_A,\rho_A)$, $B=(\lambda_B,\rho_B)$  be CIF subspaces of $\mathbf{V}'$. Then for any $\alpha,\beta\in K$, we have
$[\phi^{-1}(\alpha A_1+\beta A_2),\phi^{-1}(B)]=\alpha[\phi^{-1}(A_1),\phi^{-1}(B)]+\beta[\phi^{-1}(A_2),\phi^{-1}(B)],$
$[\phi^{-1}(A),\phi^{-1}(\alpha B_1+\beta B_2)]=\alpha[\phi^{-1}(A),\phi^{-1}(B_1)]+\beta[\phi^{-1}(A),\phi^{-1}(B_2)]$.}
\end{Theorem}
\begin{proof}
The results follow from Theorem~\ref{thrm-9}, Theorem~\ref{thrm-15} and Theorem~\ref{thrm-11}.
\end{proof}

\end{document}